\documentclass[reqno, 10pt, a4paper]{amsart}
\usepackage{amssymb,latexsym}
\usepackage{amsmath,color}
\usepackage{amsthm}
\usepackage{epsfig}
\usepackage{graphicx}
\usepackage{hyperref}
\usepackage{titletoc}
\usepackage{amsfonts}
\usepackage[numbers,sort&compress]{natbib}

\usepackage[cal=boondox]{mathalfa}

\numberwithin{equation}{section}
\newtheorem{theorem}{Theorem}[section]
\newtheorem{proposition}[theorem]{Proposition}
\newtheorem{lemma}[theorem]{Lemma}

\theoremstyle{definition}

\def\XXint#1#2#3{{\setbox0=\hbox{$#1{#2#3}{\int}$}
		\vcenter{\hbox{$#2#3$}}\kern-.5\wd0}}

\def\B{\mathbb{R}^n}
\def\R{\mathbb{R}_+^{n+1}}

\def\ou{\overline{u}}

\def\l{\lambda}
\def\s{(-\Delta)^s}
\def\ss{(-\Delta)^{\frac{s}{2}}}
\def\S{\mathbb{S}^{n-1}}
\def\e{\varepsilon}
\def\D{\Delta}
\def\vp{\varphi}
\def\RR{\mathbb{R}}
\def\ve{\varepsilon}

\begin{document}
	\title[Fractional H\'{e}non-Gelfand-Liouville]{On stable and finite Morse index solutions of the nonlocal H\'{e}non-Gelfand-Liouville equation}

	\author[M. Fazly]{Mostafa Fazly}
	\address{\noindent Mostafa Fazly, Department of Mathematics, The University of Texas at San Antonio, One UTSA Circle, San Antonio, TX, 78249}
	\email{mostafa.fazly@utsa.edu}
	
	\author[Y. Hu]{Yeyao Hu}
	\address{\noindent Yeyao Hu, School of Mathematics and Statistics, The Central South University, Changsha, Hunan 410083, P. R. China}
	\email{huyeyao@gmail.com}	
	
	\author[W. Yang]{ Wen Yang}
	\address{\noindent Wen ~Yang,~Wuhan Institute of Physics and Mathematics, Chinese Academy of Sciences, P.O. Box 71010, Wuhan 430071, P. R. China; Innovation Academy for Precision Measurement Science and Technology, Chinese Academy of Sciences, Wuhan 430071, P. R. China.}
	\email{wyang@wipm.ac.cn}

	\maketitle
	
	\begin{abstract}
		We consider the nonlocal  H\'{e}non-Gelfand-Liouville  problem
		\begin{equation*}
		(-\Delta)^s u = |x|^a e^u\quad\mathrm{in}\quad \B,
		\end{equation*} for every $s\in(0,1)$, $a>0$ and $n>2s$. We prove a monotonicity formula
		for solutions of the above equation using rescaling arguments. We apply this formula together with blow-down analysis arguments and technical integral estimates to establish non-existence of finite Morse index solutions when
		$$\dfrac{\Gamma(\frac n2)\Gamma(s)}{\Gamma(\frac{n-2s}{2})}\left(s+\frac a2\right)> \dfrac{\Gamma^2(\frac{n+2s}{4})}{\Gamma^2(\frac{n-2s}{4})}.$$
	\end{abstract}


	\section{Introduction}
		
	We study the following nonlocal H\'{e}non-Gelfand-Liouville equation 
		\begin{equation}
	\label{fg-1}
	(-\Delta)^s u = |x|^a e^u\quad\mathrm{in}\quad \B,
	\end{equation}
	when $s\in (0,1)$, $a>0$ and $n >2s$. For  $s\in(0,1)$,  the nonlocal fractional Laplacian operator $(-\D)^s $ is defined by
		\begin{equation*}
	\label{1.deff}
	\s u(x)=c_{n,s}~\mathrm{P.V.}\int_{\B}\frac{u(x)-u(y)}{|x-y|^{n+2s}}dy. 
	\end{equation*}
	Here,  P.V. stands for the principle value and $c_{n,s}$ is the normalizing constant
	$$c_{n,s}=\frac{2^{2s}}{\pi^{n/2}}\frac{\Gamma(\frac{n+2s}{2})}{|\Gamma(-s)|}.
	$$
	One can interpret equation \eqref{fg-1} in the following sense
	\begin{align}
	\label{weak-fg}
	\int_{\B}u(-\D)^s\vp dx=\int_{\B} |x|^a e^u\vp dx\quad \mbox{for every }~\vp\in C_c^\infty(\B),
	\end{align}	
	provided  that $u\in L_{s}(\B)$ and  $|x|^a e^u\in L^1_{\mathrm{loc}}(\B)$. Here the space $L_{\mu}(\B)$ for $\mu\geq-\frac{n}{2}$ is defined as
	$$L_\mu(\B):=\left\{u\in L^1_{\mathrm{loc}}(\B):\int_{\B}\frac{|u(x)|}{1+|x|^{n+2\mu}}dx<\infty  \right\}.$$
	Throughout the paper, we shall always assume $u\in L_s(\B)$  
	and $|x|^a e^u\in L^1_{\mathrm{loc}}(\B)$. Our interest is to classify the stable and finite Morse index solutions of \eqref{fg-1}. A solution $u$ to \eqref{fg-1} is said to be  stable in an open set  $\Omega\subseteq\B$ if
	\begin{equation}
	\label{1.stablecondition}
	\frac{c_{n,s}}{2}\int_{\B}\int_{\B}\frac{(\vp(x)-\vp(y))^2}{|x-y|^{n+2s}}dxdy  \geq\int_{\B}|x|^a e^u\vp^2 dx \quad \text{for every }\vp\in C_c^\infty(\Omega).
	\end{equation}
	While a solution is said to be of finite Morse index if it is stable  outside a compact set in $\B$.
	
	When $s=1,~a=0$, equation \eqref{fg-1} is reduced to the classical Liouville equation, where all the solutions with finite integral of the exponential nonlinearity in dimension two have been classified by Chen-Li in their celebrated paper \cite{cl}. They proved that up to a translation, all the solutions are radially symmetric in $\mathbb{R}^2$. The nonlocal counterpart of the Liouville equation when $2s=n=1$ is studied by Da Lio et al. in \cite{DMR}. Moreover, it has been shown that the Morse indices of these solutions are finite. In the past decade, the classification of stable and finite Morse index solutions of \eqref{fg-1} have been a very active and fruitful area in the field of nonlinear PDEs. When $a=0$, Farina in \cite{f2} and Dancer-Farina in \cite{df} established  non-existence of stable solutions to \eqref{fg-1} for $ n\leq 9$ and non-existence of finite Morse index solutions to \eqref{fg-1} for $3\leq n\leq 9$. When $a\neq 0$, Wang-Ye in \cite{wy} proved some Liouville-type theorems for weak solutions with finite Morse index. Another problem closely related to \eqref{fg-1} is the following Lane-Emden equation
	\begin{equation}
	\label{fgp-1}
	(-\Delta)^su=|x|^a|u|^{p-1}u\quad\mathrm{in}\quad \B.
	\end{equation}
	In \cite{f}, Farina obtained the optimal results for stable and finite Morse index solutions to equation (\ref{fgp-1}) when $s=1$ and $a=0$. Later on, D\'avila et al. initiated the study of equation (\ref{fgp-1}) in the nonlocal setting and they classified finite Morse index solutions of \eqref{fgp-1} when $a=0$ and $s\in(0,1)$. Later on, Wei and the first author studied higher-order nonlocal Lane-Emden equation and nonlocal H\'enon-Lane-Emden equation in \cite{fw} and \cite{fw1} respectively. Particularly, they proved a Liouville theorem for finite Morse index solutions of the nonlocal H\'{e}non-Lane-Emden equation for Sobolev subcritical exponents $1<p<p_{S}(n,a)$ and for Sobolev supercritical exponents $p_{S}(n,a)<p$ such that
	\begin{equation*} 
	\label{cond}
	p \frac{\Gamma(\frac n2-\frac{s+\frac{a}{2}}{p-1}) \Gamma(s+\frac{s+\frac{a}{2}}{p-1})}{\Gamma(\frac{s+\frac{a}{2}}{p-1}) \Gamma(\frac{n-2s}{2} - \frac{s+\frac{a}{2}}{p-1})} > \frac{\Gamma(\frac{n+2s}{4})^2}{\Gamma(\frac{n-2s}{4})^2}.
	\end{equation*}
	The above is known as the Joseph-Lundgren exponent, see \cite{jl}. Here, the Sobolev critical exponent is given by
	\begin{equation*}
	p_{S}(n,a)=\left\{
	\begin{aligned}
	&+\infty&\quad\text{if $n\le 2s$}   , \\
	&\frac{n+2s+2a}{n-2s}&\quad\text{if $n> 2s$}.
	\end{aligned}
	\right.
	\end{equation*}
	In the absence of stability notion, such classifications for Sobolev subcritical exponent were proposed by Gidas-Spruck	in the seminal work \cite{gs}, see also \cite{ps}. The literature in the context of classifying the stable solutions and finite Morse index solutions is too vast to list. For the cosmic string equation, Liouville equation, Lane-Emden equations and systems, we refer readers to \cite{ay,c1,c3,df,dggw,ddw, dgw,dyg, ddww,f,f2,fw,hx,wy,w1} and references therein.
		
It  is known, see e.g. \cite[Proposition 3.2]{r1}, that the function
\begin{equation*}
\label{1.singular}
u_{n,s}(x):=-(2s+a)\log|x|+\log \lambda_{n,s} \ \text{for}\ \lambda_{n,s}:=2^{2s}\dfrac{\Gamma(\frac n2)\Gamma(s)}{\Gamma(\frac{n-2s}{2})}\left(s+\frac a2\right),
\end{equation*}
is a singular solution to \eqref{fg-1}. Based on the following  Hardy's inequality, see \cite[Theorem 2.9]{y} and \cite{h},
	\begin{equation*}
	\frac{c_{n,s}}{2}\int_{\B}\int_{\B}\dfrac{(\psi(x)-\psi(y))^2}{|x-y|^{n+2s}}dxdy\geq\Lambda_{n,s}\int_{\B}|x|^{-2s}\psi^2dx,\quad \forall\psi\in C_c^{\infty}(\B),
	\end{equation*}
	with
    \begin{equation*}
	\label{1.hconst}
	\Lambda_{n,s}=2^{2s}\dfrac{\Gamma^2(\frac{n+2s}{4})}{\Gamma^2(\frac{n-2s}{4})}. 
	\end{equation*}
We observe that the singular solution  $u_{n,s}$ is stable if and only if
	\begin{equation}
	\label{1.stable}
	\dfrac{\Gamma(\frac n2)\Gamma(s)}{\Gamma(\frac{n-2s}{2})}\left(s+\frac a2\right)\leq \dfrac{\Gamma^2(\frac{n+2s}{4})}{\Gamma^2(\frac{n-2s}{4})},
	\end{equation}
where we use the fact that $|x|^ae^{u_{n,s}}$ is precisely the weight in Hardy's inequality. The stability condition \eqref{1.stable} for the solution $u_{n,s}$ suggests that equation \eqref{fg-1} might not admit any stable solution if the following inequality holds
	\begin{equation}
	\label{1.stable1}
	\dfrac{\Gamma(\frac n2)\Gamma(s)}{\Gamma(\frac{n-2s}{2})}\left(s+\frac a2\right)> \dfrac{\Gamma^2(\frac{n+2s}{4})}{\Gamma^2(\frac{n-2s}{4})}.
	\end{equation}
The aim of this paper is to confirm the above speculation. Our approach is based on studying the extended function $\ou$ of $u$ in the upper-half space $\R$, which is discovered by Caffarelli and Silvestre in \cite{cs}.  Let $\ou$ be the extended function of $u$ given by the Poisson formula
	\begin{equation*}
	\overline{u}(X)=\int_{\B}P(X,y)u(y)dy,\quad X=(x,t)\in \B\times (0,\infty),
	\end{equation*}
	where
	\begin{equation}
	\label{poisson}
	P(X,y)=d_{n,s}\dfrac{t^{2s}}{|(x-y,t)|^{n+2s}},
	\end{equation}
	and $d_{n,s}>0$ is a normalizing constant so that   $\int_{\B}P(X,y)dy=1$. Notice that $\ou$ is well-defined as $u\in L_s(\B)$.  Next, we define the space $\dot H^s(\Omega)$ by $$\dot H^s(\Omega):=\left\{ u\in L^2_{\mathrm{loc}}(\Omega): \int_{\Omega}\int_{\Omega}\frac{|u(x)-u(y)|^2}{|x-y|^{n+2s}}dxdy<\infty\right\}.$$
	The notation $\dot H^s_{\mathrm{loc}}(\B)$ stands for the set of  all functions  in $\dot H^s(\Omega)$ for every bounded open set $\Omega\subset\B$. It is straightforward to verify that $t^\frac{1-2s}{2}\nabla \ou\in L^{2}_{\mathrm{loc}}(\Omega\times[0,\infty))$ whenever $u\in \dot H^s(\Omega)$. Then, equation  \eqref{weak-fg} in terms of $\ou$  reads
	\begin{equation}
	\label{1.weak}
	\int_{\R}t^{1-2s}\nabla\ou\cdot\nabla\Phi(x,t)dxdt=\kappa_s\int_{\B}|x|^a e^u\vp  dx,
	\end{equation}
	for every $\Phi\in C_c^\infty(\R)$ where $\vp (x)=\Phi(x,0)$ and $\kappa_s=\frac{\Gamma(1-s)}{2^{2s-1}\Gamma(s)}.$
		
	Here is our main result concerning the non-existence of finite Morse index solution.
	
	\begin{theorem} \label{th1.1} 	
	Assume that $n>2s$ and $s\in (0,1)$. If
	\begin{equation*}
	\dfrac{\Gamma(\frac n2)\Gamma(s)}{\Gamma(\frac{n-2s}{2})}\left(s+\frac a2\right)> \dfrac{\Gamma^2(\frac{n+2s}{4})}{\Gamma^2(\frac{n-2s}{4})},
	\end{equation*}
	 holds then \eqref{fg-1} does not admit any finite Morse index solution $u\in L_s(\B)\cap W^{1,2}_{\mathrm{loc}}(\B)$ satisfying  $|x|^a e^u\in L_{\mathrm{loc}}^2(\B)$.
	\end{theorem}

Roughly speaking, the proof of Theorem \ref{th1.1} consists of two steps. First, we introduce the so-called monotonicity formula. Second, we need to express the solution of \eqref{fg-1} by the representation formula and apply it to derive some technical integral estimates to ensure that each term in the monotonicity formula is bounded. Now, let us give the monotonicity formula in the following theorem. 

   \begin{theorem}
	\label{th4.2}
Let $u\in L_{s}(\B)\cap W^{1,2}_{\mathrm{loc}}(\B)$ be a solution to  \eqref{fg-1}. Assume that    $|x|^a e^u\in L^2_{\mathrm{loc}}(\B)$. For $x_0\in\partial\R$ and $\l>0$, we define
	\begin{equation}
	\label{1.monotonicity}
		\begin{aligned}
		E(\ou,x_0,\l):=~&\lambda^{2s-n}\left(\frac12\int_{B^{n+1}(x_0,\l)\cap \R}t^{1-2s}|\nabla\ou|^2dxdt-\kappa_s\int_{B(x_0,\l)}|x|^a e^{\ou}dx\right)\\
		&+(2s+a)\l^{2s-n-1}\int_{\partial B^{n+1}(x_0,\l)\cap\R}t^{1-2s}\left(\ou+(2s+a)\log r\right)d\sigma.
		\end{aligned}
		\end{equation}
		Then,  $E$ is a nondecreasing function of $\l$. Furthermore,
		\begin{equation*}
		\frac{\partial E}{\partial\l}=\l^{2s-n}\int_{\partial B^{n+1}(x_0,\l)\cap\R}t^{1-2s}\left(\frac{\partial\ou}{\partial r}+\frac{2s+a}{r}\right)^2d\sigma,
		\end{equation*}
		where $B^{n+1}(x_0,\l)$ denotes the Euclidean ball in $\mathbb{R}^{n+1}$ centered at $x_0$ of radius $\l$, $\sigma$ is the $n$-dimensional Hausdorff measure on $\partial B^{n+1}(x_0,\l),$ $X=(x,t)\in\R$, $r=|(x-x_0,t)|$ and $\frac{\partial}{\partial r}=\nabla\cdot\frac{X-(x_0,0)}{r}$ is the corresponding radial derivative.
	\end{theorem}
The classification of stable and finite Morse index solutions using monotonicity formulas was initiated in a series of articles  \cite{ddw,ddww,w0} for the Lane-Emden equation and \cite{w3} for the Liouville type equations. We refer interested readers to \cite{fs} for an introduction regarding monotonicity formula in this context. Compared to the Lane-Emden equation, proving a priori estimates for the weighted ``Dirichlet" energy and for the boundary integral of linear term are more challenging when dealing with the Gelfand-Liouville equation. This is due to the exponential nonlinearity as opposed to the algebraic nonlinearity in these equations.  In this regard, Wang in \cite{w3},  in the study of stable solutions of Toda system,  used the $\epsilon$-regularity theory to justify the boundedness of these two terms. However, applying the same arguments in the nonlocal setting seems to be a challenging problem. Instead, Hyder and the third author in \cite{hy} developed a more straightforward method to get uniform estimates for these two integrals in the blow-down procedure, see also \cite{fwy,fy} for the recent progress on higher-order Gelfand-Liouville equation and nonlocal Toda system. Our classification result on the nonlocal H\'{e}non-Gelfand-Liouville equation relies on an iteration scheme boosting the integrability of $|x|^a e^{u}$. The revelation of a novel inequality concerning the Poisson integral of $\log{|x|}$ plays a key role, see Lemma \ref{poissonlog},  in establishing the higher-order integrability of $|x|^a e^u$ despite the difficulty caused by the H\'{e}non weighted term.  
 
The current paper is organized as follows. First, in Section \ref{Section2}, we obtain technical a priori  integral estimates concerning solutions of finite Morse index. In Section \ref{Section3}, we derive a monotonicity formula and then perform the blow-down analysis arguments. We classify homogeneous stable solutions, under the condition \eqref{1.stable1}, and apply it to prove Theorem \ref{th1.1}. In  Appendix A, we provide some technical estimates for stable solutions which are needed in our proofs.
\bigskip
	
\begin{center}
List of Notations:
\end{center}

\begin{itemize}
\item $B_R^{n+1}$ is the ball centered at $0$ with radius $R$ in dimension $(n+1)$.
\smallskip
\item $B_R$ is the ball centered at $0$ with radius $R$ in dimension $n$.
\smallskip
\item $B^{n+1}(x_0,R)$ is  the ball centered at $x_0$ with radius $R$ in dimension $(n+1)$.
\smallskip
\item $B(x_0,R)$ is the ball centered at $x_0$ with radius $R$ in dimension $n$.
\smallskip
		\item $X=(x,y)$ represents  a point in $\mathbb{R}_+^{n+1}=\B\times[0,\infty).$
		\smallskip
		\item $\ou$ is $s$-harmonic extension  of $u$ on $\mathbb{R}_+^{n+1}$.
		\item $C$ is  a generic positive constant which may vary from line to line.
		\item $R$ is  a large positive constant that may vary from line to line.
		\item $\sigma$ is the $n$-dimensional Hausdorff measure restricted to $\partial B^{n+1}(x_0,r)$.
		\item $\Omega^c$ is the complement of a subset $\Omega$ of the whole set $\B$.
		\item $\lfloor x \rfloor$ is the greatest integer smaller than or equal to $x$.
		\item $|x|\sim |y|$ means  $|x|$ is comparable with $|y|$ in a sense that $C^{-1}|y|\leq |x|\leq C|y|$ for some positive $C$.
\end{itemize}

	\section{Integral Estimates and Preliminaries}\label{Section2}
	In this section,   we use the stability condition outside a compact set to derive  energy estimates on $|x|^a e^u$  and the integral representation formula for $u$. We also give Moser iteration type estimates, see Crandall-Rabinowitz \cite{CR} and Farina \cite{f2,f}, on the weighted nonlinearity.
	
	\begin{lemma}
		\label{leh.1}
		Let $u$ be a solution to \eqref{fg-1} for some $n>2s$. Suppose that $u$ is stable outside a compact set. Then
		\begin{equation}
		\label{h.7}
		\int_{B_r}|x|^a e^udx\leq Cr^{n-2s}\quad\text{for every } r\geq 1.
		\end{equation}
	\end{lemma}	
	\begin{proof} Let $R\gg1$ be such that $u$ is stable on $\B\setminus B_R$. We introduce the following test function $\psi(x)=\eta_{2R}(x)\varphi(\frac{x}{r})$, where
	\begin{align*}
		\eta_t(x)=\left\{\begin{array}{ll}
		0\quad&\text{for }|x|\leq t\\
		\\
		1\quad&\text{for }|x|\geq 2t
		\end{array}\right.  ~\mbox{and} \quad
		\varphi(x)=\left\{\begin{array}{ll}
		1\quad&\text{for }|x|\leq 1\\
		\\
		0\quad&\text{for }|x|\geq 2  \end{array}\right..
		\end{align*}
		It is not difficult to see that $\psi$ is a good test function for the stability condition \eqref{1.stablecondition}. Hence,
		\begin{equation*}
		\int_{B(x_0,r)}|x|^a e^udx\leq C+\int_{\B}|\ss\psi|^2dx\leq C+Cr^{n-2s}\leq Cr^{n-2s},
		\end{equation*}
		where we use $n>2s$ and $r\geq 1.$
	\end{proof}
	
	It is  straightforward to see that if $u$ is a solution to \eqref{fg-1}, then
	$$u^\l(x)=u(\l x)+(2s+a)\log\l , $$
	provides a family of solutions to \eqref{fg-1}. In addition,  $u$ is stable on $\B\setminus B_R$ if and only if $u^\l$ is  stable  on $\B\setminus B_{\frac R\l}$. Due to a simple scaling argument, one can conclude that there exists $C>0$ independent of $\l$ such that
		\begin{align}
		\label{f.1}
		\int_{B_r}|x|^a e^{u^\l}dx\leq Cr^{n-2s}\quad\text{for every }r\geq 1.
		\end{align}
	We now set
	\begin{equation*}
	v^\l(x):=c(n,s)\int_{\B}\left(\frac{1}{|x-y|^{n-2s}}-\frac{1}{(1+|y|)^{n-2s}}\right)|y|^a e^{u^\l(y)}dy,
	\end{equation*}
	where $c(n,s)$ is chosen such that
	$$c(n,s)(-\Delta)^s\frac{1}{|x-y|^{n-2s}}=\delta(x-y).$$
	It is straightforward to see that $v^\l\in L^1_{\mathrm{loc}}(\B)$. In addition, we have that $v^\lambda\in L_s(\B)$. This is the conclusion of the following lemma.
	
	\begin{lemma}
		\label{lef.1}
There exists a $C>0$ independent of $\l$ such that
		\begin{equation}
		\label{f.4}
		\int_{\B}\frac{|v^\l(x)|}{1+|x|^{n+2s}}dx\leq C.
		\end{equation}
		Moreover, for $R>0$ sufficiently large, we have
		\begin{equation}
		\label{f.5}
		v^\l(x)=c(n,s)\int_{B_{2R}}\frac{|y|^a e^{u^\l(y)}}{|x-y|^{n-2s}}dy+O_R(1)\quad \mbox{for}~\l\geq 1,~|x|\leq R.
		\end{equation}
	\end{lemma}	
	\begin{proof}
		Let us first consider the term
		$$E(y):=\int_{\B}\frac{|f(x,y)|}{1+|x|^{n+2s}}dx\quad\mbox{for}\quad |y|\geq 2,$$
		where
		$$f(x,y)=\left(\frac{1}{|x-y|^{n-2s}}-\frac{1}{(1+|y|)^{n-2s}}\right).$$
	    Now, consider the following decomposition of the entire space into three parts
		$$\B=\underbrace{B_{|y|/2}}_{A_1}\cup \underbrace{B(y,|y|/2)}_{A_2}\cup \underbrace{(B_{|y|/2}\cup B(y,|y|/2))^c}_{A_3}.$$
One can verify that the following estimates hold
		\begin{equation*}
		|f(x,y)|\leq C
		\begin{cases}
		\dfrac{1+|x|}{|y|^{n-2s+1}} \quad &\mbox{if}~x\in A_1,\\
		\dfrac{1}{|y|^{n-2s}}+\dfrac{1}{|x-y|^{n-2s}} \quad &\mbox{if}~x\in A_2,\\
		\dfrac{1}{|y|^{n-2s}}  \quad &\mbox{if}~x\in A_3.
		\end{cases}
		\end{equation*}
Set
		\begin{equation*}
		E(y)=\sum_{i=1}^3 E_i(y) \ \ \text{for} \ \ E_i(y)=\int_{A_i}\dfrac{|f(x,y)|}{1+|x|^{n+2s}}dx.
		\end{equation*}
		For $E_1(y)$, one has
		\begin{equation*}
		|E_1(y)|\leq\dfrac{C}{|y|^{n-2s+1}}\int_{A_1}\dfrac{1+|x|}{1+|x|^{n+2s}}dx\leq
		\begin{cases}
		\frac{C}{|y|^{n-2s+1}},\quad &\mbox{if}~2s>1,\\
		\frac{C}{|y|^{n}}\log|y|,\quad &\mbox{if}~2s=1,\\
		\frac{C}{|y|^{n}},\quad &\mbox{if}~2s<1.
		\end{cases}
		\end{equation*}
		Next, we estimate the second term. Notice that $x\sim y$ for $x\in A_2$. Therefore,
		\begin{equation*}
		\begin{aligned}
		|E_2(y)|&\leq \frac{C}{|y|^n}+\frac{C}{|y|^{n+2s}}
		\int_{A_3}\frac{1}{|x-y|^{n-2s}}dx\\
		&\leq\frac{C}{|y|^n}+\frac{C}{|y|^{n+2s}}
		\int_{B_{3|y|}}\frac{1}{|x|^{n-2s}}dx\leq\frac{C}{|y|^n}.
		\end{aligned}
		\end{equation*}
		While for the last term $E_3(y)$, we have
		\begin{equation*}
		|E_3(y)|\leq\frac{C}{|y|^{n-2s}}\int_{\B\setminus A_1}\frac{1}{1+|x|^{n+2s}}dx\leq \frac{C}{|y|^n}.
		\end{equation*}
		Thus, there exists $\gamma>0$ such that
		\begin{equation*}
		|E(y)|\leq\dfrac{C}{|y|^{n-2s+\gamma}}\quad \mbox{for}\quad |y|\geq 2.
		\end{equation*}
		Therefore, we get
		\begin{equation}
		\label{f.111}
		\begin{aligned}
		\int_{\B}\dfrac{|v^\l(x)|}{1+|x|^{n+2s}}dx\leq~&
		C\int_{\mathbb{R}^n\setminus B_2}|y|^a e^{u^\l(y)}|E(y)|dy\\
		&+C\int_{\B}\int_{B_2}\frac{|y|^a e^{u^\l(y)}}{|x-y|^{n-2s}(1+|x|^{n+2s})}dydx\\
		&+C\int_{\B}\int_{B_2}\frac{|y|^a e^{u^\l(y)}}{((1+|y|)^{n-2s})(1+|x|^{n+2s})}dydx.
		\end{aligned}
		\end{equation}
		One can  check that the latter two terms in (\ref{f.111}) are bounded, so it remains to show that
\begin{equation*}
\begin{aligned}
&\int_{\mathbb{R}^n\setminus B_2}|y|^a e^{u^\l(y)}|E(y)|dy \leq C  \int_{\mathbb{R}^n\setminus B_2}|y|^a e^{u^\l(y)} \frac{1}{|y|^{n-2s+\gamma}}dy\\
&\leq C \sum\limits_{i=1}^{\infty} \int_{2^i \leq |y|\leq 2^{i+1}} \frac{|y|^a e^{u^{\l}}}{|y|^{n-2s+\gamma}}dy\leq C\sum\limits_{i=1}^{\infty} \frac{1}{2^{i\gamma}}<\infty.
\end{aligned}
\end{equation*}		
		This finishes the proof of \eqref{f.4}. In order to prove \eqref{f.5}, we have
		\begin{equation*}
		\begin{aligned}
v^{\l}(x) =~& c(n,s)\int_{\B}\left(\frac{1}{|x-y|^{n-2s}}-\frac{1}{(1+|y|)^{n-2s}}\right) |y|^a e^{u^{\l}}dy\\
          =~& c(n,s)\int_{\B\setminus B_{2R}}\left(\frac{1}{|x-y|^{n-2s}}-\frac{1}{(1+|y|)^{n-2s}}\right) |y|^a e^{u^{\l}}dy\\
           ~&+ c(n,s)\left(\int_{B_{2R}}	\frac{|y|^a e^{u^{\l}}}{|x-y|^{n-2s}}dy - \int_{B_{2R}} \frac{|y|^a e^{u^{\l}}}{(1+|y|)^{n-2s}}dy\right)\\
          =~& c(n,s)\int_{B_{2R}}	\frac{|y|^a e^{u^{\l}}}{|x-y|^{n-2s}}dy+ O_R(1) ,
		\end{aligned}
		\end{equation*}
where we use that		
		\begin{equation*}
		\left|\frac{1}{|x-y|^{n-2s}}-\frac{1}{(1+|y|)^{n-2s}}\right|\leq C\frac{1+|x|}{|y|^{n-2s+1}} \quad\text{for }|x|\leq R,\, |y|\geq 2R.
		\end{equation*}
Hence,  we finish the proof.
	\end{proof}
Here, we provide an integral representation of $u^\l(x)$. This plays an essential role in our proof of main results.
	
	\begin{lemma}\label{u-representation}For some $c_\lambda\in\RR$,   we have
		\begin{equation}
		\label{f.11}
		u^\l(x)=c(n,s)\int_{\B}\left(\frac{1}{|x-y|^{n-2s}}-\frac{1}{(1+| y|)^{n-2s}}\right)|y|^a e^{u^\l(y)}dy+c_\l.
		\end{equation}
	\end{lemma}
    \begin{proof} 
	We define the difference of $u^\l$ and $v^\l$ by $h^\l$. Based on the definition of $v^\l$, we have $h^\lambda$ is a $s$-harmonic function in $\B$. By \cite[Lemma 2.4]{ha} one can get that   $h^\l $ is either a constant, or   a polynomial of degree one. To verify that $h^\l$ is constant,   we first claim that
		\begin{equation}
		\label{f.10}
		v^\l(x)\geq -C\log|x|\quad\mbox{for}~|x|~\mbox{large},
		\end{equation} for some constant $C>0$.
		Indeed,  by \eqref{f.1} we have
		\begin{align*}
		v^\l(x)&\geq -C\left(1+\int_{2\leq|y|\leq2|x|}\frac{1}{(1+|y|)^{n-2s}}
		|y|^a e^{u^\l(y)}dy\right)\\
		&\geq -C-C\sum_{i=1}^{\lfloor\log_2(2|x|)\rfloor}\int_{2^i\leq |y|\leq 2^{i+1}}\frac{1}{(1+|y|)^{n-2s}}|y|^a e^{u^\l(y)}dy\\
		&\geq -C- C\log|x|.
		\end{align*}
 Hence, \eqref{f.10} is proved and $u^\l(x)\geq -C\log|x| +h^\l(x).$ Using \eqref{f.1} again, we derive that $h^\l  \equiv c_\l$ and \eqref{f.11} holds.
	\end{proof}

We now establish higher-order integrability estimates for the finite Morse index solutions of \eqref{fg-1}. The methods and ideas applied here are motivated by the ones given in \cite{f,f2,df,CR,dggw,hy,fw}. We present the iteration scheme in this section and, for the convenience of readers, leave some necessary steps to Appendix A. We overcome the major difficulty due to the H\'{e}non weight by absorbing it into the exponential so that Jensen's inequality applies.	Before proving the integral estimate, we need the following lemma concerning the Poisson integral of $\log{|x|}$. This may be known to the analysis community but we could not find a reference for it, so we prove it here. 
\begin{lemma}
\label{poissonlog}
Let $P(X,y)$ be as defined in \eqref{poisson} when $X=(x,t)\in\mathbb R^{n+1}_+$. For any $|x|,t>0$, we have
\begin{equation}
\label{logext}
\int_{\B} P(X,y)\log{|y|}dy\geq \log{|x|}.
\end{equation}
\end{lemma}
\begin{proof}
The difference of two sides of \eqref{logext} yields that
\begin{equation}
\label{logext1}
\begin{aligned}
&\int_{\B} P(X,y)\left(\log{|y|}-\log{|x|}\right)dy\quad \\
&= d(n,s)\int_{\B} \frac{t^{2s}|x|^n}{(|x-|x|z|^2+t^2)^{\frac{n+2s}{2}}}\log{|z|}dz~\quad (\mbox{where}~z=y/|x|)\\
&= C\tilde{t}^{2s} \int_{\B} \frac{\log{|z|}}{(|\theta-z|^2+\tilde{t}^2)^{\frac{n+2s}{2}}}dz~\quad  (\mbox{where}~\theta=x/|x|~\mathrm{and}~\tilde t=t/|x|)\\
&= C \tilde{t}^{2s} \left(\int_{B_1}+\int_{\B\setminus B_1}\right)\frac{\log{|z|}}{(|\theta-z|^2+\tilde{t}^2)^{\frac{n+2s}{2}}}dz\\
&= C \tilde{t}^{2s} \int_{B_1}\left(\frac{1}{(|\theta-z|^2+\tilde{t}^2)^{\frac{n+2s}{2}}}-\frac{1}{(|\theta-\frac{z}{|z|^2}|^2+\tilde{t}^2)^{\frac{n+2s}{2}}|z|^{2n}}\right)\log{|z|} dz\\
&= C \tilde{t}^{2s} \int_{ B_1}\left(\frac{1}{(|\theta-z|^2+\tilde{t}^2)^{\frac{n+2s}{2}}}-\frac{|z|^{2s-n}}{(|\theta-z|^2+|z|^2\tilde{t}^2)^{\frac{n+2s}{2}}}\right)\log{|z|} dz,
\end{aligned}
\end{equation}
where we use the facts that $|\theta-\frac{z}{|z|^2}|^2=\frac{|\theta-z|^2}{|z|^2}$ for any $\theta\in \mathbb{S}^{n-1}$ in the last step. Using $n>2s$, one can  verify that
\begin{align*}
&\frac{1}{(|\theta-z|^2+\tilde{t}^2)^{\frac{n+2s}{2}}}-\frac{|z|^{2s-n}}{(|\theta-z|^2+|z|^2\tilde{t}^2)^{\frac{n+2s}{2}}}\leq 0\quad \mathrm{for}~|z|\leq1.
\end{align*}
Together with the fact $\log|z|\leq 0$ for $z\leq1$, we conclude  that the right-hand side of \eqref{logext1} is non-negative and this proves \eqref{logext}.
\end{proof}

Note that if $n\geq 2$, one can provide an alternative proof of \eqref{logext}.  We have 
\begin{equation}
\label{logext2}
\begin{aligned}
&\int_{\B} P(X,y)\left(\log{|y|}-\log{|x|}\right)dy\quad \\
&=\int_{0}^\infty\int_{\partial B(x,\vartheta)}P((x-z,t))(\log|x-z|-\log|x|)
d\sigma d\vartheta.
\end{aligned}
\end{equation}
It is known that 
$$\int_{\partial B(x,\vartheta)}\log|x-z|d\sigma\geq\int_{\partial B(x,\vartheta)}\log|x|d\sigma, $$
due to the fact that $\log|x|$ is subharmonic when $n\geq2.$ Hence the right-hand side of \eqref{logext2} is non-negative and \eqref{logext} is proved. The same kind of inequality should hold for all subharmonic functions including but not limited to $\log{|x|}$. It provides an angle to view \eqref{logext} as a consequence of mean value inequality of subharmonic function. We wonder whether there is a specific interpretation of the inequality \eqref{logext} if $n\in (2s,2)$. 

Now, we are ready to state and prove the main result of this section. 
\begin{proposition}
		\label{prop-2.6}
		Let $u\in L_s(\B)\cap \dot H^s_{\mathrm{loc}}(\B)$ be a solution to \eqref{fg-1}. Assume that $u$ is stable on $\B\setminus B_R$ for some $R>0$. Then, for every $p\in [1,\min\{5,1+\frac{n}{2s}\})$ there exists $C=C(p)>0$ such that for $r$ large
		\begin{align}
		\label{est-p-1}
		\int_{B_{2r}\setminus B_{r}}\left(|x|^a e^{u}\right)^p dx\leq Cr^{n-2ps}.
		\end{align}
In particular,
		\begin{itemize}
			\item[(i)] for $|x|$ large and for  every  $p\in[1,\min\{5,1+\frac{n}{2s}\})$
			\begin{align}
			\label{est-p-2} \int_{B(x,|x|/2)}\left(|y|^a e^{u (y)}\right)^{p} dy\leq C(p)|x|^{n-2ps},
			\end{align}
			\item[(ii)] for $r$ large and for  every $p\in[1,\min\{5, \frac{n}{2s}\})$
			\begin{align} \label{est-p-3}\int_{B_r\setminus  B_{2R}}\left(|x|^a e^{u}\right)^p dx\leq C(p)r^{n-2ps} .
			\end{align}
		\end{itemize}
	\end{proposition}
	\begin{proof}
By the Moser's interation, see Lemma \ref{lem-2.7} and Lemma \ref{lem-0.2} in Appendix A, one can show that $|x|^a e^{u}\in L^p_{\mathrm{loc}}(\B\setminus B_R)$ for every $p\in [1,5)$. To prove \eqref{est-p-1} we first show it holds for $p\in[0,1)$. Indeed, by the H\"older's inequality and Lemma \ref{leh.1} we get that for $p\in[0,1)$
		\begin{equation*}
		\begin{aligned}
		\int_{B_{2r}\setminus B_r}\left(|x|^a e^{u}\right)^p dx \leq \left(\int_{B_{2r}\setminus B_r}|x|^a e^u dx\right)^p\left(\int_{B_{2r}\setminus B_r}1 dx\right)^{1-p}\leq Cr^{n-2p s}.
		\end{aligned}
		\end{equation*}
		Next we claim that if
		\begin{align}
		\label{claim-2.31}
		\int_{B_{2r}\setminus B_{r}}\left(|x|^a e^{u}\right)^{p}dx\leq Cr^{n-2p s}\quad \text{for  }r>2R,
		\end{align}
		for some $\alpha\in (0,\min\{\frac{n}{4s},2\})$, then
		\begin{align}
		\label{claim-2.32}\int_{B_{2r}\setminus B_{r}}\left(|x|^a e^{u}\right)^{p+1}dx\leq Cr^{n-2(p+1) s}\quad\text{for }r>3R.
		\end{align}
		From \eqref{claim-2.31}, it is straightforward to show that
		\begin{equation}
		\label{3.claim}
		\int_{B_r\setminus B_{2R}}\left(|x|^a e^{u}\right)^{p}dx\leq Cr^{n-2p s}\quad \text{for  }r> 2R.
		\end{equation}
		Indeed, for  $r>2R$ of the form $r=2^{N_1}$ with some positive integer $N_1$,  and taking  $N_2$ to be the smallest integer such that $2^{N_2}\geq 2R$,  by  \eqref{claim-2.31} we deduce
		\begin{equation}
		\label{3.claim1}
		\begin{aligned}
		\int_{B_{2^{N_1}}\setminus B_{2R}}\left(|x|^a e^{u}\right)^{p}dx=
		&\left(\int_{B_{2^{N_2}}\setminus B_{2R}}+\sum_{\ell=1}^{N_1-N_2}\int_{B_{2^{N_2+\ell}}\setminus B_{2^{N_1+\ell}}}\right)\left(|x|^a e^{u}\right)^{p}dx\\
		\leq & ~C+ C\sum_{\ell=1}^{N_1-N_2}(2^{N_2+\ell})^{n-2p s}
		\leq  C2^{N_1(n-2p s)},
		\end{aligned}
		\end{equation}
		where we use $n-2p s>0$. Then, using the hypothesis \eqref{claim-2.31},   we derive the following decay estimate
		\begin{equation}
		\label{h.11}
		\begin{aligned}
		\int_{|y|\geq r}\frac{\left(|x|^a e^{u(x)}\right)^{p}}{|x|^{n+2s}}dx=~&
		\sum_{i=0}^\infty\int_{2^{i+1}r\geq |x|\geq 2^ir}\frac{|x|^{pa}e^{pu(x)}}{|x|^{n+2s}}dx\\
		\leq~&\frac{ C}{r^{(2p+2)s}}\sum_{i=0}^\infty\frac{1}{2^{(2p+2)si}}
		\leq Cr^{-(2p+2)s}.
		\end{aligned}
		\end{equation}
		Next, we shall show similar estimates as \eqref{h.11} hold for the extended function $\ou$. 
		For $r>3R$, $x\in B_{3r}\setminus B_{2r/3}$ and $t\in(0,2r)$,  we first apply the Jensen's inequality to establish the following 
\begin{equation}
\label{jensen}
\begin{aligned}
&\int_{\B\setminus B_{2R}} P(X,y)\left(|y|^a e^{u(y)}\right)^{p} dy + \int_{B_{2R}}P(X,y)dy\\
&= \int_{\B\setminus B_{2R}} P(X,y) e^{p (u(y)+a\log{|y|})}dy + \int_{B_{2R}}P(X,y)dy\\
&\geq e^{p\left(\int_{\B}P(X,y)(u(y)+a\log{|y|})dy
	-\int_{B_{2R}}P(X,y)(u(y)+a\log{|y|})dy\right) }\\
&\geq C e^{p\ou(x,t)} e^{p a \int_{\B}P(X,y)\log{|y|}}dy
\geq C |x|^{pa}e^{p\ou(x,t)},
\end{aligned}
\end{equation}
where we use 
\begin{equation*}
\int_{B_{2R}}\left(u(y)+a\log{|y|}\right) P(X,y)dy\leq C,
\end{equation*}
together with Lemma \ref{poissonlog} and $a>0$.		
We now give an estimate based on \eqref{3.claim}, \eqref{h.11} and \eqref{jensen} 
\begin{equation}
\label{weighted}
\begin{aligned}
&\int_{B_{3r}\setminus B_{2r/3}}|x|^{pa} e^{p\ou(x,t)}dx\\ &\leq C\int_{B_{3r}\setminus B_{2r/3}} \int_{\B\setminus B_{2R}}P(X,y)|y|^{p a} e^{p u(y)}dy dx+C\int_{B_{3r}\setminus B_{2r/3}}\int_{B_{2R}}P(X,y)dydx\\
&\leq C\int_{B_{3r}}\left(\int_{B_{4r}\setminus B_{2R}}
+\int_{\B\setminus B_{4r}}\right)
P(X,y)|y|^{p a} e^{p u(y)} dy dx+C R^n\\
&\leq C\int_{B_{4r}\setminus B_{2R}}|y|^{p a} e^{p u(y)} \left(\int_{B_{3r}}P(X,y)dx\right)dy+ C \int_{B_{3r}} \frac{t^{2s}}{r^{(2p+2)s}} dx +C R^n\\
&\leq  C\int_{B_{4r}\setminus B_{2R}}|y|^{pa} e^{p u(y)}dy+ C r^{n-2p s} +C R^n\\
&\leq Cr^{n-2p s} + C R^n \leq C r^{n-2ps},
\end{aligned}
\end{equation}		
where we use $n-2p s>0$.

Now,  we fix non-negative smooth functions $\vp$ on $\B$ and $\eta$ on $[0,\infty)$ such that $\vp(x)=1$ on $B_{2}\setminus B_1$, $\vp(x)=0$ on $B_{2/3}\cup B_{3}^c$; $\eta(t)=1$ on $[0,1]$, $\eta(t)=0$ on $[2,\infty)$. For $r>0$, we set $\Phi_r(x,t)=|x|^{\frac{p}{2}a}\vp(\frac xr)\eta(\frac tr)$. Then  $\Phi_r$ is a good test function in Lemma \ref{lem-0.2}. Therefore, as  $|\nabla \Phi_r |\leq  \frac Cr |x|^{\frac{p}{2}a}$, by \eqref{weighted} we obtain
		\begin{equation*}
		\begin{aligned}
		\int_{\R}t^{1-2s}e^{p\ou}|\nabla  \Phi_r |^2dxdt
		&\leq Cr^{-2}\int_0^{2r} t^{1-2s}\int_{B_{3r}\setminus B_{2r/3}}\left(|x|^a e^{\ou(x,t)}\right)^{p}dxdt\\
		&\leq Cr^{n-2-2p s}\int_0^{2r}t^{1-2s} dt\leq Cr^{n-2s(1+p)}.
		\end{aligned}
		\end{equation*}
Similarly, we have
		\begin{equation*}
		\left|\int_{\R}e^{p\ou} \nabla\cdot [t^{1-2s}\nabla  \Phi_r^2]dxdt\right|\leq Cr^{n-2s(1+p)}.
		\end{equation*}		
		Then,  \eqref{claim-2.32} follows from Lemma \ref{lem-0.2}. Repeating the above arguments finitely many times we get \eqref{est-p-1}, while \eqref{est-p-2} follows immediately as $B(x,|x|/2)\subset B_{2r}\setminus B_{r/2}$ with $r=|x|$.
		
		In spirit of the estimate   \eqref{3.claim1},  one can obtain  the second conclusion \eqref{est-p-3}.  This  finishes the  proof.
	\end{proof}

	\section{Monotonicity Formula and Blow-down Analysis}\label{Section3}
In this section, we provide a proof of our main result. We first present a monotonicity formula for solutions to the extension problem and derive several energy estimates to the finite Morse index solutions by the results stated in Section \ref{Section2}. Then, we show the non-existence of stable homogeneous solutions under the condition \eqref{1.stable1}.  Finally, we prove that the limit of the blow-down sequence is homogeneous and stable in $\mathbb{R}^n\setminus\{0\}$, which gives a contradiction and it proves Theorem \ref{th1.1}.
\subsection{Monotonicity Formula}
	In this subsection  we  give a proof of Theorem \ref{th4.2}.
	\begin{proof}[Proof of Theorem \ref{th4.2}.]

We prove the theorem  for $u$  sufficiently smooth and give the necessary details without the smoothness assumption in the end of  proof. Without loss of generality, we assume that $x_0=0$. Set
		\begin{equation*}
		\begin{aligned}
		E_1(\ou,\lambda)=\lambda^{2s-n}\Bigg(\frac12\int_{B^{n+1}_{\l}\cap\R}t^{1-2s}|\nabla\ou|^2 dxdt-\kappa_s\int_{B^{n+1}_{\l}\cap\partial\R}|x|^a e^{\ou}dx\Bigg).
		\end{aligned}
		\end{equation*}
		Define
		\begin{equation*}
		\ou^\lambda (X)=\ou(\lambda X)+(2s+a)\log\lambda.
		\end{equation*}
		Then
		\begin{equation}
		\label{3.relation}
		E_1(\ou,\lambda)=E_1(\ou^{\lambda},1).
		\end{equation}
		Differentiating $\ou^\l$ with respect to $\lambda $, we have
		\begin{equation*}
		\lambda\frac{\partial\ou^{\l}}{\partial\lambda}=r\frac{\partial \ou^{\l}}{\partial r}+2s+a.
		\end{equation*}
		Differentiating the right-hand side of \eqref{3.relation}, we find
		\begin{equation}
		\label{3.diff}
		\begin{aligned}
		\frac{\partial E_1(\ou,\lambda)}{\partial\lambda}=~&\int_{B^{n+1}_1\cap\R}t^{1-2s}\nabla \ou^{\lambda}\nabla \frac{\partial\ou^{\l}}{\partial\lambda}dxdt
		-\kappa_s\int_{B^{n+1}_1\cap\partial\R}e^{\ou^{\lambda}}\frac{\partial\ou^{\lambda}}{\partial\l}dx\\
		=~&\int_{\partial B^{n+1}_1\cap\R}t^{1-2s}\frac{\partial\ou^\lambda}{\partial r}\frac{\partial \ou^\lambda}{\partial\l} d\sigma\\
		=~&\lambda \int_{\partial B^{n+1}_1\cap\R}\left(t^{1-2s}\left(\frac{\partial\ou^\lambda}{\partial\l}\right)^2
		-(2s+a)t^{1-2s}\frac{\partial \ou^\lambda}{\partial\l}\right)d\sigma.
		\end{aligned}
		\end{equation}
		We notice that
		\begin{equation}
		\label{3.m2}
		E(\ou,\lambda)=E(\ou^\lambda,1)=E_1(\ou^\lambda,1)+(2s+a)\int_{\partial B^{n+1}_1\cap\R}t^{1-2s}\ou^\lambda d\sigma.
		\end{equation}
		From \eqref{3.diff} and \eqref{3.m2}, we get
		\begin{equation*}
		\begin{aligned}
		\frac{\partial E(\ou,\lambda)}{\partial\lambda}=~&\lambda \int_{\partial B^{n+1}_1\cap\R}t^{1-2s}\left(\frac{\partial\ou^\lambda}{\partial\l}\right)^2d\sigma\\
		=~&\lambda^{2s-n}\int_{\partial B^{n+1}_{\l}\cap\R}t^{1-2s}\left(\frac{\partial\ou}{\partial r}+\frac{2s+a}{r}\right)^2d\sigma.
		\end{aligned}
		\end{equation*}
		This proves Theorem \ref{th4.2} when the function $u^\l$ is smooth.  For general $u^\l$ we set $u^\lambda_{\e}:=u^\lambda*\rho_{\e}$, $(\rho_{\e})_{\e>0}$ are the standard mollifiers, it is straightforward to see that $u^\lambda_{\e}$ satisfies $\s u^\lambda_{\e}=(|x|^ae^{u^\lambda})*\rho_{\e}$. Then,  we consider $E(\ou^\lambda_{\e},1)$. Following the same computations as above, we could get
\begin{equation*}
\begin{aligned}
\frac{\partial E(\ou^\lambda_{\e},1)}{\partial\lambda}~=~&\l\int_{\partial B^{n+1}_1\cap\R}t^{1-2s}\left(\frac{\partial\ou^\l_\e}{\partial \l}\right)^2d\sigma\\
&+\kappa_s\int_{B_1}\left((|x|^ae^{u^\l})*\rho_{\e}-|x|^ae^{u_\e^\l}\right)\frac{\partial u_\e^\l}{\partial\l} dx,
\end{aligned}
\end{equation*}
where the second term on the right hand side could be controlled by assuming $u\in W^{1,2}_{\mathrm{loc}}(\B)$ and $e^{u}\in L^2_{\mathrm{loc}}(\B)$, and it converges to zero as $\e$ tends  to $0$. Hence we finish the proof.
\end{proof}
	
\subsection{Energy Estimates} We begin by stating the following lemma for the third term in the monotonicity formula \eqref{1.monotonicity}.
		
		\begin{lemma}
		\label{lef.2}
		Let $\ou^\l$ be the $s$-harmonic extension of $u^\l$, then
		\begin{equation}
		\label{3.boundary-1}
		\int_{\partial B_1^{n+1}\cap\R}t^{1-2s}\ou^\l(X)d\sigma=c_sc_\l+O(1),
		\end{equation}
		and
		\begin{equation}
		\label{3.area-1}
		\int_{B_r^{n+1}\cap\R}t^{1-2s}\ou^\l(X)dxdt=\omega_sc_\l r^{n+2-2s}+C(r),
		\end{equation}
		where $c_\l$ is defined in \eqref{f.11}, $c_s$ and $\omega_s$ are  positive finite numbers given by
		$$c_s:=\int_{\partial B_1^{n+1}\cap\R}t^{1-2s}d\sigma\quad\mathrm{and}\quad \omega_s=\int_{B_1^{n+1}\cap\R}t^{1-2s}dx dt.$$
	    \end{lemma}
	\begin{proof}
	Since the proof of \eqref{3.boundary-1} and \eqref{3.area-1} are almost the same, we  only provide the details for the former one.	Using the Poisson formula and \eqref{f.4},  we have
		\begin{equation*}
		\begin{aligned}
		&\int_{\partial B^{n+1}_1\cap\R}t^{1-2s}\ou^\l(X)d\sigma=\int_{\partial B^{n+1}_1\cap\R}t^{1-2s}\int_{\B}P(X,z)u^\l(z)dzd\sigma\\
		&=\int_{\partial B^{n+1}_1\cap\R}t^{1-2s}c_\l d\sigma+\int_{\partial B^{n+1}_1\cap\R}t^{1-2s}\int_{\B}P(X,z)v^\l(z)dz\\
		&=\int_{\partial B^{n+1}_1\cap\R}t^{1-2s}c_\l d\sigma+\int_{\partial B^{n+1}_1\cap\R}t^{1-2s}\left(\int_{\B\setminus B_2}+\int_{B_2}\right)P(X,z)v^\l(z)dz\\
		&= c_sc_\l+O(1)+\int_{\partial B_1^{n+1}\cap\R}t^{1-2s}\int_{B_2}P(X,z)v^\l(z)dzd\sigma.
		\end{aligned}
		\end{equation*}
		We denote the last term in the above equation by $\Xi$. Concerning the term $\Xi$, we have that
		\begin{equation}
		\label{II}
		\begin{aligned}
		|\Xi|&\leq C+C\int_{\partial B^{n+1}_1\cap\R}t^{1-2s}\int_{|z|\leq 2}P(X,z)\int_{|y|\leq 4}\frac{|y|^a e^{u^\l(y)}}{|z-y|^{n-2s}}dydzd\sigma\\
		&\leq C+C\int_{|y|\leq 4}|y|^a e^{u^\l(y)}\int_{\partial B_{1}^{n+1}\cap\R}\int_{|z|\leq 2}t^{1-2s}P(X,z)\frac{1}{|z-y|^{n-2s}}dzd\sigma dy, 
		\end{aligned}
		\end{equation}
		where we use \eqref{f.1} and \eqref{f.5}. To show that $\Xi$ is uniformly bounded with respect to $\l$, it suffices to establish the following claim
		\begin{equation}
		\label{f.14}
		\int_{\partial B^{n+1}_1\cap\R}\int_{B_4}t^{1-2s}P(X,z)\frac{1}{|y-z|^{n-2s}}dzd\sigma\leq C\quad \mbox{for every}~y\in\B.
		\end{equation}
		Indeed, for $x\neq y$ we set $\rho=\frac12|x-y|$. Then we have
		\begin{equation}
		\label{f.15}
		\begin{aligned}
		&\int_{B_4}\frac{t}{|(x-z,t)|^{n+2s}|y-z|^{n-2s}}dz\\
		&\leq\left(\int_{B(y,\rho)}+\int_{B_4\setminus  B(y,\rho)}\right)\frac{t}{|(x-z,t)|^{n+2s}|y-z|^{n-2s}}dz\\
		&\leq C\frac{t}{(\rho+t)^{n+2s}}
		\int_{B(y,\rho)}\frac{1}{|y-z|^{n-2s}}dz
		+\frac{1}{\rho^{n-2s}}\int_{B_4\setminus  B(y,\rho)}\frac{t}{|(x-z,t)|^{n+2s}}dz\\
		&\leq C\left(\frac{t\rho^{2s}}{(\rho+t)^{n+2s}}
		+\frac{t^{1-2s}}{\rho^{n-2s}}\right)\leq C\left(\frac{1}{\rho^{n-1}}+\frac{t^{1-2s}}{\rho^{n-2s}}\right).
		\end{aligned}
		\end{equation}
We use the stereo-graphic projection $(x,t)\to\xi$ from $\partial B_1^{n+1}\cap\R\to\mathbb{R}^n\setminus B_1$, i.e.,
		$$\xi\to(x,t)=\left(\frac{2\xi}{1+|\xi|^2},~\frac{|\xi|^2-1}{1+|\xi|^2}\right).$$
		$$(x,t)\to\xi=\frac{x}{1-t}.$$
		Then, $\rho=\frac12\left|\frac{2\xi}{1+|\xi|^2}-y\right|$ and it follows that
		\begin{equation}
		\label{f.16}
		\int_{\partial B_1^{n+1}\cap\R}\frac{1}{\rho^{n-1}}d\sigma
		\leq\int_{|\xi|\geq 1}\frac{1}{\rho^{n-1}}\frac{1}{(1+|\xi|^2)^n}d\xi\leq C,
		\end{equation}
		and
		\begin{equation}
		\label{f.17}
		\begin{aligned}
		\int_{\partial B_1^{n+1}\cap\R}\frac{t^{1-2s}}{\rho^{n-2s}}d\sigma
		\leq \int_{|\xi|\geq1}\frac{t^{1-2s}}{\rho^{n-2s}}
		\frac{1}{(1+|\xi|^2)^n}d\xi\leq C.
		\end{aligned}
		\end{equation}
		From \eqref{II}-\eqref{f.17},  the proof is completed.
	\end{proof}

In order to estimate the first term in the monotonicity formula \eqref{1.monotonicity}  we need the following result.
	
	\begin{lemma}
		\label{lef.3}
		There exists a constant $C>0$ such that
		$$\int_{B_r}|(-\D)^\frac s2u^\lambda(x)|^2dx\leq Cr^{n-2s}\quad\text{for~ }r\geq1,~\ \lambda\geq1.$$
	\end{lemma}
	\begin{proof}
		We  prove the lemma for $\lambda=1$. The general case follows from a  rescaling argument. From \eqref{f.11}, we know that
		$$(-\D)^\frac s2u(x)=C\int_{\B}\frac{|y|^a e^{u(y)}}{|x-y|^{n-s}}dy, $$
		in the sense of distribution. Suppose $u$ is stable in $\B\setminus B_R$, we decompose $B_{r}$ as the union of $B_{4R}$ and $B_{r}\setminus B_{4R}$ for $r\gg R.$ Since  $u\in\dot H_{\mathrm{loc}}^s(\B)$,  we get
		\begin{equation}
		\label{4.4.0}
		\int_{B_{4R}}|(-\D)^\frac{s}{2} u|^2dx\leq C(R).
		\end{equation}
		In $B_r\setminus B_{4R}$, we have
		\begin{equation}
		\label{4.4.1}
		\begin{aligned}
		|(-\D)^\frac s2u(x)| &\leq C\left(\int_{B_{2R}}    +\int_{B_{2r}\setminus B_{2R}}+\int_{\B\setminus B_{2r}}\right)
		\frac{|y|^a e^{u(y)}}{|x-y|^{n-s}}dy\\
		&\leq  \frac{C}{|x|^{n-s}}+C\int_{B_{2r}\setminus B_{2R}}\frac{|y|^a e^{u(y)}}{|x-y|^{n-s}}dy+ C\int_{\B\setminus B_{2r}}
		\frac{|y|^a e^{u(y)}}{|y|^{n-s}}dy \\
		&=:C\left(\frac{1}{|x|^{n-s}}+I_1(x)+I_2 \right) .
		\end{aligned}
		\end{equation}
		For the last term, using \eqref{h.7} we get
		\begin{equation}
		\label{I2}
		\begin{aligned}
		I_2=\sum_{k=1}^\infty \int_{r2^k\leq|x|\leq r2^{k+1}}\frac{|y|^a e^{u(y)}}{|y|^{n-s}}dy  \leq
		C\sum_{k=1}^\infty \frac{(2^{k+1}r)^{n-2s}}{(2^kr)^{n-s}}\leq \frac{C}{r^{s}}.
		\end{aligned}
		\end{equation}
		Therefore,
		\begin{equation}
		\label{4.4.2}
		\int_{B_r}I_2^2dx\leq Cr^{n-2s}.
		\end{equation}
		For the second term $I_1(x)$, we make a further decomposition
		\begin{equation}
		\label{newpf}
		\begin{aligned}
		I_1 &= \left(\int_{(B_{2|x|}\setminus B_{2R})\cap B(x,\frac{|x|}{2})}+\int_{(B_{2|x|}\setminus B_{2R})\setminus B(x,\frac{|x|}{2})}+\int_{B_{2r}\setminus B_{2|x|}}\right) \frac{|y|^a e^{u(y)}dy}{|x-y|^{n-s}}\\
		&= I_{1,1}+ I_{1,2}+ I_{1,3}.
		\end{aligned}
		\end{equation}
For the terms $I_{1,2}$ and $I_{1,3}$, by \eqref{f.1} and \eqref{I2} we have
\begin{equation}
\label{I13-I12}
I_{1,2}+I_{1,3}\leq \frac{C}{|x|^s}.
\end{equation}		
While for the left term $I_{1,1}$, we use Proposition \ref{prop-2.6} together with the H\"older's inequality to get the following estimates
\begin{equation}
\label{newpfHolder}
\begin{aligned}
\left(\int_{B(x,\frac{|x|}{2})}\frac{|y|^a e^{u(y)}}{|x-y|^{n-s}}dy\right)^2 &\leq \int_{B(x,\frac{|x|}{2})}\frac{dy}{|x-y|^{n-s}} \int_{B(x,\frac{|x|}{2})}\frac{|y|^{2a} e^{2u(y)}}{|x-y|^{n-s}}dy\\
&\leq C|x|^s \int_{B(x,\frac{|x|}{2})}\frac{|y|^{2a} e^{2u(y)}}{|x-y|^{n-s}}dy.
\end{aligned}
\end{equation}
Combining \eqref{newpf}-\eqref{newpfHolder}, we conclude that
\begin{equation}
\label{4.4.3}
\begin{aligned}
\int_{B_{r}\setminus B_{4R}} I^2_1 dx &\leq Cr^{n-2s} + C\int_{B_r\setminus B_{4R}}|x|^s \int_{B(x,\frac{|x|}{2})}\frac{|y|^{2a} e^{2u(y)}}{|x-y|^{n-s}}dy dx\\
&\leq Cr^{n-2s} + C \int_{B_{2r}\setminus B_{2R}} |y|^{2a} e^{2u(y)}\left(\int_{B(\frac{4y}{3},\frac{2|y|}{3})}\frac{|x|^s}{|x-y|^{n-s}}dx\right)dy\\
&\leq C r^{n-2s} + C\int_{B_{3r/2}\setminus B_{2R}}|y|^{2(s+a)} e^{2u(y)}dy\leq C r^{n-2s},
\end{aligned}
\end{equation}
where the last inequality follows from the same argument in proving \eqref{3.claim1} and the fact that $n>2s$.

From \eqref{4.4.1}, \eqref{4.4.2} and \eqref{4.4.3}, we deduce that
		\begin{equation}
		\label{4.4.4}
		\int_{B_r\setminus B_{4R}}|(-\D)^\frac{s}{2} u|^2dx\leq C r^{n-2s}.
		\end{equation}
Then,  the lemma follows from \eqref{4.4.0}, \eqref{4.4.4} and $n>2s$.
	\end{proof}

	We use Lemma \ref{lef.3} to prove the following estimate.
	\begin{lemma} There exists a constant $C$ such that
		\label{lef.4}
		\begin{equation}
		\label{f.24}
		\int_{B^{n+1}_r\cap \R}t^{1-2s}|\nabla \ou^\l(X)|^2dxdt\leq C(r),\quad \l\geq1.
		\end{equation}
	\end{lemma}
	\begin{proof}
	We introduce a cut-off function $\varphi\in C_c^\infty(B_{4r})$ with $\varphi=1$ in $B_{2r}$. Writing $u^\l$ into 
	$$u^\l=u_1^\l+u_2^\l,$$ where
		\begin{align*}
		u^\l_1(x)&=c(n,s)\int_{\B}\left(\frac{1}{|x-y|^{n-2s}}-\frac{1}{(1+|y|)^{n-2s}}\right)\varphi(y)|y|^ae^{u^\l(y)} dy+c_\l,
		\end{align*}
		and
		\begin{align*}
		u^\l_2(x)&=c(n,s)\int_{\B}\left(\frac{1}{|x-y|^{n-2s}}-\frac{1}{(1+|y|)^{n-2s}}\right)(1-\varphi(y))|y|^ae^{u^\l(y)} dy.
		\end{align*}
		Let $\ou_i^\l$ denote the $s$-harmonic extension of $u_i^\l$ for $i=1,2$ respectively. As in Lemma \ref{lef.3}, one can show that    
		\begin{equation}
		\label{f.25}
		\int_{\R}t^{1-2s}|\nabla\ou^\l_1(x)|^2dxdt=\kappa_s\int_{\B}\left|\ss u^\l_1(x)\right|^2dx\leq C(r).
		\end{equation}
		Next, we shall prove that
		\begin{equation}
		\label{f.26}
		\int_{B^{n+1}_r\cap\R}t^{1-2s}|\nabla\ou^\l_2(x)|^2dxdt\leq C(r)\quad \mbox{for every}~r\geq 1,~\lambda\geq1.
		\end{equation}
		Proceeding similarly as in the proof of Lemma \ref{lef.1}, one can verify that
		\begin{equation}
		\label{f.27}
		\int_{\B}\frac{|u_{2}^\l(x)|}{1+|x|^{n+2s}}dx\leq C(r)\quad \mathrm{and}\quad \|\nabla u_2^\l\|_{L^\infty(B_{3r/2})}\leq C(r).
		\end{equation}
As a consequence, we have
		\begin{equation}
		\label{f.28}
		\|u_2^\l\|_{L^\infty(B_{3r/2})}\leq C(r).
		\end{equation}
		In order to prove \eqref{f.26}, we  consider $\partial_t\ou^\l_2$ and $\nabla_x\ou^\l_2$ separately. For the first term, we notice that
		\begin{align*}
		\partial_t\ou^\l_2(X)=~&\partial_t(\ou^\l_2(x,t)-u^\l_2(x))=d_{n,s}\partial_t\int_{\B}\frac{t^{2s}}{|(x-y,t)|^{n+2s}}(u_2^\l(y)-u_2^\l(x))dy\\
		=~&d_{n,s}\int_{\B}\partial_t\left(\frac{t^{2s}}{|(x-y,t)|^{n+2s}}\right)(u_2^\l(y)-u_2^\l(x))dy,
		\end{align*}
		where we use
		$$d_{n,s}\int_{\B}\frac{t^{2s}}{|(x-y,t)|^{n+2s}}dy=1.$$
		From \eqref{f.27} and \eqref{f.28},  for $|x-x_0|\leq r$ we get
		\begin{equation}
		\label{f.29}
		\begin{aligned}
		&\left|\int_{\B\setminus B_ {3r/2}}\partial_t\left(\frac{t^{2s}}{|(x-y,t)|^{n+2s}}\right) (u_2^\l(y)-u_2^\l(x)) dy\right|\\
		&\leq C(r)t^{2s-1}\int_{\B\setminus B_ {3r/2}}\frac{(|u_2^\l(y)|+1)}{1+|y|^{n+2s}}dy\leq C(r)t^{2s-1}.
		\end{aligned}
		\end{equation}
		Using \eqref{f.27}, we have
		\begin{equation}
		\label{f.30}
		\begin{aligned}
		&\left|\int_{B_{3r/2}}\partial_t\left(\frac{t^{2s}}{|(x-y,t)|^{n+2s}}\right)(u_2^\l(y)-u_2^\l(x))dy\right|\\
		&\leq C\|\nabla u_2^\l\|_{L^\infty(B_{3r/2})}\int_{B_{3r/2}}\left|\partial_t\left(\frac{t^{2s}}{|(x-y,t)|^{n+2s}}\right)\right||x-y|dy\\
		&\leq C(r)\left(t^{2s-1}+1\right).
		\end{aligned}
		\end{equation}
		By \eqref{f.29} and \eqref{f.30}, we obtain
		\begin{equation}
		\label{f.31}
		\begin{aligned}
		\int_{B^{n+1}_r\cap\R}t^{1-2s}|\partial_t\ou^\l_2|^2dxdt
		&\leq C(r)\int_{B^{n+1}_r\cap\R}\left(t^{1-2s}+t^{2s-1}\right)dxdt\\
		&\leq C(r).
		\end{aligned}
		\end{equation}
	    For the term $\nabla_x\ou_2^\l$,  in a similar way we have
		\begin{equation}
		\label{f.35}
		\int_{B^{n+1}_r\cap\R}t^{1-2s}|\nabla_x\ou^\l_2|^2dxdt\leq C(r).
		\end{equation}
		Then,  \eqref{f.24} follows from \eqref{f.25}, \eqref{f.31} and \eqref{f.35}. Hence, the proof is completed.
	\end{proof}

	\begin{proposition}
		\label{prf.1}
There exists a constant $C$ such that $c_\lambda\leq C$ for $\lambda\in [1,\infty)$. Consequently, we have
		\begin{equation*}
		\lim_{\l\to+\infty}E(\ou,\l)=\lim_{\l\to\infty}E(\ou^\l,1)<+\infty.
		\end{equation*}
	\end{proposition}
	\begin{proof} Combining \eqref{f.1} and Lemma \ref{lef.4} implies that the following term
	\begin{equation*}
	\frac12\int_{ B^{n+1}_1\cap\R}t^{1-2s}|\nabla\ou^\l|^2dxdt-\kappa_s\int_{B^{n+1}_1\cap\partial\mathbb{R}_+^{n+1}}|x|^a e^{\ou^\l}dx
	\end{equation*}
	is bounded in $\l\in[1,\infty)$.
	Applying Theorem \ref{th4.2} and Lemma \ref{lef.2}, we get
	\begin{equation*}
	E(\ou^\l,1)=(2s+a)c_s c_\l+O(1)\geq E(\ou,1)=(2s+a)c_sc_1+O(1),
	\end{equation*}
	which implies that
	\begin{equation*}
	c_\l ~\ \mbox{is bounded from below for}~\ \l\geq1.
	\end{equation*}
	From \eqref{f.5}, we have $$u^\l=v^\l+c_\l\geq c_\l-C~\ \text{in}~\  B_R~\ \text{for}~\ \l\geq 1.$$ 
	Then, by \eqref{f.1} we get
	 $$c_\l\leq C\quad\text{for }\l\geq 1.$$
	Thus,  we conclude  that $c_\l$ is bounded. Therefore, Lemma \ref{lef.2} implies that
	\begin{equation*}
	\int_{\partial B^{n+1}_1\cap\R}t^{1-2s}\ou^\l(X)d\sigma=O(1).
	\end{equation*}
This finishes the proof.
 \end{proof}
	
We end this part with an estimate that will be used in the proof of Theorem \ref{th1.1}.

\begin{lemma}\label{lem-H1bound} For every $r>1$ and $\lambda\geq 1$ we have $$\int_{B_r^{n+1}\cap\R}t^{1-2s}\left(|\ou^\l|^2+|\nabla \ou^\l|^2\right)dxdt\leq C(r).$$  \end{lemma}

\begin{proof}
We apply the Poincar\'{e} inequality with $A_2$ Muckenhoupt weight $t^{1-2s}$, see \cite{Fabes,HKP,PRP}, to get 
\begin{equation*}
\int_{B_r^{n+1}\cap\R}t^{1-2s} |\ou^{\l}-\mbox{avg}(\ou^{\l})|^2 dxdt \leq C(r)\int_{B_r^{n+1}\cap\R}t^{1-2s} |\nabla\ou^{\l}|^2 dxdt,
\end{equation*}
where $$\mbox{avg}(\ou^{\l})=\frac{\int_{B_r^{n+1}\cap\R}t^{1-2s} \ou^{\l}dxdt}{\omega_s r^{n+2-2s}},$$ denotes the average of $u^{\l}$ over the upper half ball. Therefore, the proof follows from Lemma \ref{lef.2} and Lemma \ref{lef.4}. \end{proof}

	\subsection{Non-existence of Stable Homogeneous Solution}
	In this part, we  prove the non-existence of stable homogeneous solutions of the form of $\tau(\theta)-(2s+a)\log r$. The arguments in this section are motivated by the ones given in \cite[section 4.1]{hy} and in \cite{ddw,ddww,fw,fwy} for the classification of  stable homogeneous solutions. 
	\begin{theorem}
		\label{th3.1}
		There is no stable solution of \eqref{fg-1} of the form $\tau(\theta)-(2s+a)\log r$ provided
		\begin{equation*}
		\dfrac{\Gamma(\frac n2)\Gamma(s)}{\Gamma(\frac{n-2s}{2})}\left(s+\frac a2\right)> \dfrac{\Gamma^2(\frac{n+2s}{4})}{\Gamma^2(\frac{n-2s}{4})}.
		\end{equation*}
		Here,  $\theta=\frac{x}{|x|}\in\mathbb{S}^{n-1}.$
	\end{theorem}
	
	\begin{proof}
		For any radially symmetric function $\vp\in C_c^\infty(\B)$, we have
		\begin{equation*}
		\int_{\B}\tau(\theta)\s\vp dx=0.
		\end{equation*}
		Then, from \eqref{weak-fg}
		\begin{equation*}
		\begin{aligned}
		\int_{\B}e^{\tau(\theta)-2s\log|x|}\vp dx
		=~&\int_{\B}\left[\tau(\theta)-(2s+a)\log|x|\right]\s\vp dx\\
		=~&\int_{\B}\left[-(2s+a)\log|x|\right]\s\vp dx\\=~&\frac{2s+a}{2s}A_{n,s}\int_{\B}\frac{\vp}{|x|^{2s}}dx,
		\end{aligned}
		\end{equation*}
		where  the last equality follows from  (see e.g. \cite{r1})
		$$\s\log\frac{1}{|x|^{2s}}=A_{n,s}\frac{1}{|x|^{2s}}=2^{2s}\frac{\Gamma(\frac n2)\Gamma(1+s)}{\Gamma(\frac{n-2s}{2})}\frac{1}{|x|^{2s}}.$$
		This leads to
		\begin{equation*}
		\begin{aligned}
		0&=\int_{\B}\left(e^{\tau(\theta)}-\frac{2s+a}{2s}A_{n,s}\right)\frac{\vp}{|x|^{2s}}dx\\
		&=\int_0^\infty r^{n-1-2s}\vp(r)\int_{\mathbb{S}^{n-1}}\left(e^{\tau(\theta)}-\frac{2s+a}{2s}A_{n,s}\right)d\theta dr,
		\end{aligned}
		\end{equation*}
		which implies
		\begin{equation}
		\label{3.5}
		\int_{\mathbb{S}^{n-1}}e^{\tau(\theta)}d\theta=\frac{2s+a}{2s}A_{n,s}|\mathbb{S}^{n-1}|.
		\end{equation}
		Now,  we shall use the stability  condition to derive a counterpart equation of \eqref{3.5}.  We fix a  radially symmetric smooth   cut-off function
		\begin{equation*}
		\eta(x)=\begin{cases}
		1\quad  &\mathrm{for}~|x|\leq1,\\ \rule{0cm}{.5cm}
		0\quad  &\mathrm{for}~|x|\geq2,
		\end{cases}
		\end{equation*}
		and set	
		$$\eta_{\varepsilon}(x)=\left(1-\eta\left(\frac{2x}{\e}\right)\right)\eta(\e x).$$
		It is straightforward to see that $\eta_{\e}=1$ for $\e<r<\e^{-1}$ and $\eta_{\e}=0$ for either $r<\frac{\e}{2}$ or $r>\frac{2}{\e}$. We test the stability condition \eqref{1.stablecondition} on the function $\psi(x)=r^{-\frac{n-2s}{2}}\eta_{\e}(r)$. Let $|y|=rt$ and we notice that
		\begin{equation*}
		\begin{aligned}
		\int_{\B}\frac{\psi(x)-\psi(y)}{|x-y|^{n+2s}}dy
		&= r^{-\frac n2-s}\int_0^\infty\int_{\S}\frac{\eta_{\e}(r)-t^{-\frac{n-2s}{2}}\eta_{\e}(rt)}{(t^2+1-2t\langle \theta,\omega\rangle)^{\frac{n+2s}{2}}}t^{n-1}dtd\omega\\
		&= r^{-\frac n2-s}\eta_{\e}(r)\int_0^{\infty}\int_{\S}\frac{1-t^{-\frac{n-2s}{2}}}
		{(t^2+1-2t\langle\theta,\omega\rangle)^{\frac{n+2s}{2}}}t^{n-1}dtd\omega\\
		&\quad +r^{-\frac n2-s}\int_0^{\infty}\int_{\S}\frac{t^{n-1-\frac{n-2s}{2}}(\eta_{\e}(r)-\eta_{\e}(rt))}
		{(t^2+1-2t\langle\theta,\omega\rangle)^{\frac{n+2s}{2}}}dtd\omega.
		\end{aligned}
		\end{equation*}
		It is known that, see e.g.  \cite[Lemma 4.1]{fa}, 
		\begin{equation*}
		\Lambda_{n,s}=c_{n,s}\int_0^{\infty}\int_{\S}\frac{1-t^{-\frac{n-2s}{2}}}
		{(t^2+1-2t\langle\theta,\omega\rangle)^{\frac{n+2s}{2}}}t^{n-1}dtd\omega.
		\end{equation*}
		Therefore,
		\begin{equation*}
		\begin{aligned}
		c_{n,s}\int_{\B}\frac{\psi(x)-\psi(y)}{|x-y|^{n+2s}}dy
		=~&c_{n,s}r^{-\frac n2-s}\int_0^{\infty}\int_{\S}\frac{t^{n-1-\frac{n-2s}{2}}(\eta_{\e}(r)-\eta_{\e}(rt))}
		{(t^2+1-2t\langle\theta,\omega\rangle)^{\frac{n+2s}{2}}}dtd\omega\\
		&+\Lambda_{n,s}r^{-\frac n2-s}\eta_{\e}(r).
		\end{aligned}
		\end{equation*}
		Based on the above computations, we compute the left-hand side of the stability inequality \eqref{1.stablecondition} 
		\begin{equation}
		\label{3.8}
		\begin{aligned}
		&\frac{c_{n,s}}{2}\int_{\B}\int_{\B}\frac{(\psi(x)-\psi(y))^2}{|x-y|^{n+2s}}dxdy\\
		&=c_{n,s}\int_{\B}\int_{\B}\frac{(\psi(x)-\psi(y))\psi(x)}{|x-y|^{n+2s}}dxy\\
		&=c_{n,s}\int_0^\infty dt\left[\int_0^\infty r^{-1}\eta_{\e}(r)(\eta_{\e}(r)-\eta_{\e}(rt))dr\right]\int_{\S}\int_{\S}\frac{t^{n-1-\frac{n-2s}{2}}d\omega d\theta }
		{(t^2+1-2t\langle\theta,\omega\rangle)^{\frac{n+2s}{2}}}\\
		&\quad+\Lambda_{n,s}|\S|\int_0^\infty r^{-1}\eta_{\e}^2(r)dr.
		\end{aligned}
		\end{equation}
		We compute the right-hand side of  the stability inequality \eqref{1.stablecondition} for the test function $\psi(x)=r^{-\frac{n}{2}+s}\eta_{\e}(r)$ and $u(r)=-(2s+a)\log r+\tau(\theta)$,
		\begin{equation}
		\label{3.9}
		\begin{aligned}
		\int_{\B}|x|^a e^u\psi^2 dx=&\int_0^\infty\int_{\S}\eta_{\e}^2(r)r^{-2s}r^{-(n-2s)}e^{\tau(\theta)}r^{n-1}drd\theta\\
		=&\int_0^\infty r^{-1}\eta^2_{\e}(r)dr\int_{\S}e^{\tau(\theta)}d\theta.
		\end{aligned}
		\end{equation}
		From the definition of function $\eta_{\e}$, we have
		\begin{equation*}
		\int_0^\infty r^{-1}\eta_{\e}^2(r)dr=\log\frac{2}{\e}+O(1).
		\end{equation*}
		One can see that both the first term on the right-hand side of \eqref{3.8} and the right-hand side of \eqref{3.9} carry the term $\int_0^\infty r^{-1}\eta_{\e}^2(r)dr$ and it tends to $\infty$ as $\e\to0$. As \cite{hy,fw}, one can show that
		\begin{equation*}
		f_{\e}(t):=\int_0^\infty r^{-1}\eta_{\e}(r)(\eta_{\e}(r)-\eta_{\e}(rt))dr =O(\log t).
		\end{equation*}
		Consequently,  we have
		\begin{equation*}
		\begin{aligned}
		&\int_0^{\infty}\left[\int_0^\infty r^{-1}\eta_{\e}(r)(\eta_{\e}(r)-\eta_{\e}(rt))\right]\int_{\S}\int_{\S}
		\dfrac{t^{n-1-\frac{n-2s}{2}}}
		{(t^2+1-2t\langle\theta,\omega\rangle)^{\frac{n+2s}{2}}}d\omega d\theta dt\\
		&\approx \int_0^\infty\int_{\S}\int_{\S}\dfrac{t^{n-1-\frac{n-2s}{2}}\log t}
		{(t^2+1-2t\langle\theta,\omega\rangle)^{\frac{n+2s}{2}}}d\omega d\theta dt\\
		& = O(1) .
		\end{aligned}
		\end{equation*}
		Collecting the higher-order term $(\log\e)$, we get
		\begin{equation}
		\label{3.12}
		\Lambda_{n,s}|\S|\geq \int_{\S}e^{\tau(\theta)}d\theta.
		\end{equation}
		From \eqref{3.5} and \eqref{3.12}, we obtain that
		\begin{equation*}
		\Lambda_{n,s}\geq \frac{2s+a}{2s} A_{n,s},
		\end{equation*}
		which contradicts the assumption \eqref{1.stable1}. Therefore, such homogeneous solution does not exist and we finish the proof.
	\end{proof}

 	\subsection{Blow-down Analysis} \label{section5}
	In this part, we prove Theorem \ref{th1.1}.
	
	\begin{proof}[Proof of Theorem \ref{th1.1}.]
	Let  $u$  be  a finite Morse index solution to \eqref{fg-1} for some  $n>2s$ satisfying  \eqref{1.stable1}. Let $R>1$ be such  that $u$   is stable outside the ball  $B_R$.
		
		From Lemma \ref{lem-H1bound},  we obtain that there exists a sequence $\lambda_i\to+\infty$ such that $\ou^{\lambda_i}$ converges weakly in $\dot{H}^1_{\mathrm{loc}}(\overline\R,t^{1-2s}dxdt)$ to a function $\ou^\infty$. Due to Lemma \ref{lef.3},  the sequence $\{u^{\l}\}$ has a weak limit $u^{\infty}$ and $u^{\l}\rightarrow u^{\infty}$ strongly in $L^2_{\mathrm{loc}}(\B)$. To show that $u^{\infty}$ satisfies \eqref{weak-fg}, we need to establish the uniform integrability of $u^{\l}$ with respect to a function decays at infinity like $|x|^{-(n+2s)}$, i.e., for any $\ve$ there exists $r\gg1$ such that
		\begin{equation}
		\label{5.1}
		\int_{\B\setminus B_r}\frac{|u^\l(x)|}{1+|x|^{n+2s}}dx<\ve,\quad\text{for every }\lambda\geq 1.
		\end{equation}
		In fact,  as in \eqref{f.111}, we get
		\begin{equation*}
		\begin{aligned}
		\int_{\B\setminus B_r}\dfrac{|u^\l(x)|}{1+|x|^{n+2s}}dx\leq~&
		C\int_{\B\setminus B_r}\dfrac{c_\lambda}{1+|x|^{n+2s}}dx
		+C\int_{\mathbb{R}^n\setminus B_{r_0}}|y|^a e^{u^\l(y)}|E(y)|dy\\
		&+C\int_{\B\setminus B_r}\int_{B_{r_0}}
		\frac{|y|^a e^{u^\l(y)}}{|x-y|^{n-2s}(1+|x|^{n+2s})}dydx\\
		&+C\int_{\B\setminus B_r}\int_{B_{r_0}}
		\frac{|y|^a e^{u^\l(y)}}{((1+|y|)^{n-2s})(1+|x|^{n+2s})}dydx\\
		\leq ~& Cr^{-2s}+Cr_0^{-\gamma}+Cr^{-2s}r_0^{n-2s},
		\end{aligned}
		\end{equation*}
where the uniform boundedness of $c_{\l}$ achieved in Proposition \ref{prf.1} is used. We  first choose $r_0$ large enough such that $Cr_0^{-\gamma}\leq\ve/2$ and then choose $r$ such that $Cr^{-2s}r_0^{n-2s}+Cr^{-2s}\leq\ve/2$. Thus, \eqref{5.1} is proved. As a consequence,  $u^\infty\in L_s(\B)$, and for any $\vp\in C_c^\infty(\B)$
		\begin{equation}
		\label{5.2}
		\lim\limits_{i\to\infty}\int_{\B}u^{\lambda_i}\s\vp dx=\int_{\B}u^{\infty}\s\vp dx.
		\end{equation}
By \eqref{5.2}, we have that $u^{\infty}$ is indeed the restriction of $\ou^{\infty}$ on $\B$. On the other hand, we can also show that
		\begin{equation}
		\label{5.3}
		\lim\limits_{i\to\infty}\int_{\B}|x|^a e^{u^{\lambda_i}}\vp dx=\int_{\B}|x|^a e^{u^{\infty}}\vp dx,\quad \forall \varphi\in C_c^\infty(\B).
		\end{equation}
		From \eqref{5.2} and \eqref{5.3},  we conclude that $u^\infty$ satisfies equation \eqref{weak-fg}.

		Now,   we show that the limit function $\ou^\infty$ is homogenous, and is of the form $-(2s+a)\log r+\tau(\theta)$. First of all, for any $r>0$ it is easy to see that
		\begin{equation}
		\label{5.r}
		{\lim_{i\to\infty}E(\ou,\l_i r)~\ \mbox{is independent of}~\ r}.
		\end{equation}
	    Applying \eqref{5.r} for $R_2>R_1>0,$ we get
		\begin{equation*}
		\begin{aligned}
		0=~&\lim_{i\to\infty}E(\ou,\l_i R_2)-\lim_{i\to+\infty}E(\ou,\l_i R_1)\\
		=~&\lim_{i\to\infty}E(\ou^{\l_i}, R_2)-\lim_{i\to\infty}E(\ou^{\l_i},R_1)\\
		\geq~&\lim_{i\to\infty}\inf\int_{\left(B^{n+1}_{R_2}\setminus B^{n+1}_{R_1}\right)\cap\R}t^{1-2s}r^{2s-n}\left(\frac{\partial\ou^{\l_i}}{\partial r}+\frac{2s+a}{r}\right)^2dxdt\\
		\geq~&\int_{\left(B^{n+1}_{R_2}\setminus B^{n+1}_{R_1}\right)\cap\R}t^{1-2s}r^{2s-n}\left(\frac{\partial\ou^\infty}{\partial r}+\frac{2s+a}{r}\right)^2dxdt.
		\end{aligned}
		\end{equation*}
		Notice that in the last inequality we only used the weak convergence of $\ou^{\l_i}$
		to $\ou^\infty$ in $H^1_{\mathrm{loc}}(\overline\R,t^{1-2s}dxdt)$. So,
		$$\frac{\partial\ou^\infty}{\partial r}+\frac{2s+a}{r}=0\quad \mbox{a.e. in}\quad \R.$$
		Thus,  the claim is proved. In addition, $u^\infty$ is stable because the stability condition for $u^{\l_i}$ passes to the limit. Then, from Theorem \ref{th3.1} we conclude that \eqref{1.stable} holds, which contradicts \eqref{1.stable1}. This completes the proof.
	\end{proof}

\appendix
\renewcommand{\appendixname}{Appendix~\Alph{section}}

\section{}

The purpose of this appendix is to provide necessary details for the proof of Proposition \ref{prop-2.6}. First, we notice that the stability condition \eqref{1.stablecondition} can be extended to  $\ou$.	More precisely, if $u$ is stable in $\Omega$ then
	\begin{equation}
	\label{h.2}
	\int_{\R}t^{1-2s}|\nabla\Phi|^2dxdt\geq \kappa_s\int_{\B}|x|^a e^u\vp ^2dx,
	\end{equation}
	for every $\Phi\in C_c^\infty(\overline\R)$ satisfying $\vp(\cdot):=\Phi(\cdot,0)\in C_c^\infty(\Omega)$.
	Indeed, if  $\overline{\vp }$ is the $s$-harmonic extension of $\vp $,   we have
	\begin{align*}
	\int_{\R}t^{1-2s}|\nabla\Phi|^2dxdt &\geq \int_{\R}t^{1-2s}|\nabla\overline\vp |^2dxdt\\&
	=\kappa_s\int_{\B}\vp (-\Delta)^s\vp  dx\geq \kappa_s\int_{\B}|x|^a e^u\vp ^2dx.
	\end{align*}
	
	\medskip
	Before giving the conclusion of Moser's iteration, we present the following lemma which will be used in the later context.
	
	\begin{lemma}
		\label{lem-2.7}Let $e^{\alpha u}\in L^1( \Omega) $ for some $\Omega\subset \B$. Then $t^{1-2s}e^{\alpha\ou}\in L^1_{\mathrm{loc}}(\Omega\times [0, \infty))$.
	\end{lemma}
	\begin{proof}
		Let $\Omega_0\Subset\Omega$ be fixed.   Since $ u\in  L_s(\B)$, we have for $x\in\Omega_0$ and $t\in(0,R)$
		$$\ou(x,t)\leq C+\int_{\Omega}u(y)P(X,y)dy=C+\int_{\Omega}g(x,t)u(y)\frac{P(X,y)dy}{g(x,t)},$$
		where
		$C \leq g(x,t):=\int_{\Omega}P(X,y)dy\leq 1$ for some positive constant $C$ depending solely on $R$, $\Omega_0$ and $\Omega$. Therefore, by the Jensen's inequality
		\begin{align*}
		&\int_{\Omega_0}e^{\alpha \ou(x,t)}dx
		\leq C \int_{\Omega_0}\int_{\Omega}e^{\alpha g(x,t)u(y)}P(X,y)dydx\\
		&\leq C\int_{\Omega}\max\{e^{\alpha u(y)},1\}\int_{\Omega_0}P(X,y)dxdy
		\leq C+C\int_{\Omega}e^{\alpha u(y)}dy,
		\end{align*}
		where the constant $C$ depends on $R,~\Omega_0$ and $\Omega$, but not on $t$. Hence,
		\begin{align*}
		\int_{\Omega_0\times(0,R)} t^{1-2s}e^{\alpha \ou(x,t)}dxdt \leq
		\int_0^R t^{1-2s}\int_{\Omega_0}e^{\alpha \ou(x,t)}dxdt <\infty.
		\end{align*}
		This finishes the proof.
	\end{proof}

	The following lemma is the main conclusion of this appendix and it is essential in the proof of Proposition \ref{prop-2.6}.
	
	\begin{lemma}
		\label{lem-0.2}
		Let $u\in L_s(\B)\cap \dot H^s_{\mathrm{loc}}(\B)$ be a solution to \eqref{fg-1}. Assume that $u$ is stable in $\Omega\subseteq\B$. Let    $\Phi \in C_c^\infty({\overline\R})$ be of the form $\Phi (x,t) =\vp(x)\eta(t)$  for some $\vp\in C_c^\infty(\Omega)$  and $\eta \equiv 1$ on $[0,1]$.  Then for every $0<\alpha<2$  we have
		\begin{equation*}
		\begin{aligned}
		(2-\alpha)\kappa_s\int_{\B}|x|^a e^{(1+2\alpha) u}\vp^2dx
		&\leq 2\int_{\R}t^{1-2s}e^{2\alpha\bar u}|\nabla \Phi|^2dxdt\\
		&\quad  -\frac12\int_{\R}e^{2\alpha\bar u}\nabla\cdot[t^{1-2s}\nabla\Phi^2]dxdt . \end{aligned}
		\end{equation*}
	\end{lemma}

	\begin{proof}
		For $k\in\mathbb{N}$ we set $\ou_k:=\min\{\ou,k\}$. Let $u_k$ be the restriction of $\ou_k$ on $\B$.  It is straightforward to see that $e^{2\alpha\ou_k}\Phi^2$ can be considered as a  test function in \eqref{1.weak}. Therefore,
		\begin{equation}
		\begin{aligned}
		\label{7}
		&\kappa_s\int_{\B}|x|^a e^u e^{2\alpha u_k}\vp^2dx\\
		&=2\alpha\int_{\R}t^{1-2s} e^{2\alpha\ou_k}\Phi^2\nabla\ou\cdot \nabla \ou_kdxdt +\int_{\R}t^{1-2s}e^{2\alpha\ou_k}\nabla\ou\cdot \nabla \Phi ^2dxdt\\
		&=2\alpha\int_{\R}t^{1-2s} e^{2\alpha\ou_k}\Phi^2\ | \nabla \ou_k|^2dxdt +\int_{\R}t^{1-2s}e^{2\alpha\ou_k}\nabla\ou\cdot \nabla \Phi^2dxdt.
		\end{aligned}
		\end{equation}
		Now,  we assume that $t^{1-2s}e^{2(\alpha+\ve)\ou}\in L^1_{\mathrm{loc}}(\Omega\times [0,\infty))$ for some $\ve>0$. Then by \cite[Lemma 3.5]{hy}, up to a subsequence, we have 
		$$\kappa_s\int_{\B}|x|^a e^u e^{2\alpha u_k}\vp^2dx=(2\alpha+o(1))\int_{\R}t^{1-2s} e^{2\alpha\ou_k}\Phi^2\ | \nabla \ou_k|^2dxdt+O(1).$$
		Taking $e^{\alpha\ou_k} \Phi$  as a test function in the stability inequality  \eqref{h.2} yields 
		\begin{equation}
		\label{8}
		\begin{aligned}
		&\kappa_s \int_{\B}|x|^a e^ue^{2\alpha u_k}\vp^2dx\\
		 &\leq \alpha^2 \int_{\R}t^{1-2s}\Phi^2e^{2\alpha\ou_k}|\nabla\ou_k|^2dxdt + \int_{\R}t^{1-2s}e^{2\alpha\ou_k}|\nabla  \Phi|^2dxdt\\
		  &\quad+\frac12 \int_{\R}t^{1-2s} \nabla e^{2\alpha \ou_k}\nabla\Phi ^2dxdt\\
		&=  \alpha^2 \int_{\R}t^{1-2s}\Phi^2e^{2\alpha \ou_k}|\nabla \ou_k|^2dxdt + \int_{\R}t^{1-2s}e^{2\alpha\ou_k}|\nabla \Phi|^2dxdt
		\\ &\quad-\frac12 \int_{\R}e^{2\alpha\ou_k} \nabla\cdot [t^{1-2s}\nabla \Phi^2]dxdt,
		\end{aligned}
		\end{equation}
		where the last equality follows from integration by parts. Notice that the boundary term vanishes as $\eta(t)=1$ on $[0,1]$. Combining the prior estimates, we obtain
		\begin{equation}
		\label{9}
		\begin{aligned}
		&(2-\alpha-\ve)\kappa_s\int_{\B}|x|^a e^ue^{2\alpha u_k}\vp^2dx
		\\&\leq(-2\alpha\ve+o(1))\int_{\R}t^{1-2s} e^{2\alpha\ou_k}\Phi^2|\nabla \ou_k|^2dxdt+O(1)
		\\&\quad+2\int_{\R}t^{1-2s}e^{2\alpha\ou_k}|\nabla \Phi|^2dx -\int_{\R}e^{2\alpha\ou_k}\nabla\cdot[t^{1-2s}\nabla \Phi^2]dxdt.
		\end{aligned}
		\end{equation}
		Concerning the last term in \eqref{9}, one can notice that 
		$$\nabla\cdot[t^{1-2s}\nabla \Phi^2]=t^{1-2s}\eta^2\D_x\vp^2+\vp^2\partial_t(t^{1-2s}\partial_t\eta^2).$$
		Again, as $\eta=1$ on $[0,1]$,  the second term   in the right-hand side of the  above expression is identically zero for $0\leq t\leq 1$. Therefore,  Lemma \ref{lem-2.7} yields 
		$$\left|\int_{\R}t^{1-2s}e^{2\alpha\ou_k}|\nabla \Phi|^2dxdt\right| +\left|\int_{\R}e^{2\alpha\ou_k}\nabla\cdot[t^{1-2s}\nabla \Phi^2]dxdt\right|\leq C.$$
		Thus,
		\begin{align*}
		(2-\alpha-\ve)\int_{\B}e^ue^{2\alpha u_k}\vp^2dx\leq C,
		\end{align*}
		provided $\int_{\Omega}e^{2(\alpha+\ve)u}dx<\infty.$ Now, choosing $\alpha\in(0,\frac12)$ and $0<\ve<\frac 12-\alpha$ in the above estimate and then taking $k\to\infty$ we conclude $e^{(1+2\alpha)u}\in L^1_{\mathrm{loc}}(\Omega)$. By an iteration argument we conclude that $e^{(1+2\alpha)u}\in L^1_{\mathrm{loc}}(\Omega)$ for every $\alpha\in(0,2)$.
		
Now, send $k\to\infty$ in \eqref{9} to  get 
		$$\int_{\R}t^{1-2s}e^{2\alpha\ou}\Phi^2|\nabla\ou|^2dxdt<\infty.$$ 
		And, take limit in \eqref{7} and \eqref{8} as $k\to\infty$. Then, the proof of lemma follows immediately since the second term on the right-hand side of \eqref{7}, when  $k\to\infty$, can be re-written as
		\begin{align*}
		\int_{\R}t^{1-2s}e^{2\alpha\ou}\nabla \ou\cdot\nabla\Phi^2dxdt
		&=\frac{1}{2\alpha} \int_{\R}t^{1-2s}\nabla e^{2\alpha\ou}\cdot\nabla \Phi^2dxdt\\
		&=-\frac{1}{2\alpha} \int_{\R} e^{2\alpha\ou}\nabla\cdot[t^{1-2s}\nabla \Phi^2]dxdt.
		\end{align*} Note that the boundary integral is zero as $\eta=1$ on $[0,1]$. This completes the proof. 
	\end{proof}

\section*{Acknowledgement}

	 The research of the third author is partially supported by NSFC No.11801550 and NSFC No.11871470. The third author thanks Ali Hyder for many stimulating discussions.


\begin{thebibliography}{99}
		\bibitem{ay} W.W. Ao and W. Yang, On the classification of solutions of cosmic strings equation. {\em Ann. Mat. Pura Appl.} (4) 198 (2019), no. 6, 2183-2193.	
		
		\bibitem{cs} L. Caffarelli and L. Silvestre, An extension problem related to the fractional Laplacian.
		{\em Comm. Partial Differential Equations} 32, 7-9 (2007), 1245-1260.	
		
		\bibitem{cl} W. X. Chen and C. M. Li, Classification of solutions of some nonlinear elliptic equations. {\em Duke Math. J.}~ 63 (1991), no. 3, 615-622.
		
		\bibitem{c1} C. Cowan, Liouville theorems for stable Lane-Emden systems with biharmonic problems. {\em Nonlinearity} 26, 8 (2013), 2357-2371.
		
		
		\bibitem{c3} C. Cowan and M. Fazly, On stable entire solutions of semi-linear elliptic equations with weights. {\em Proc. Amer. Math. Soc.} 140, 6 (2012), 2003-2012.
		
		
\bibitem{CR} M. G. Crandall and P. H. Rabinowitz, Some continuation and
variation methods for positive solutions of nonlinear elliptic eigenvalue
problems. {\em Arch. Rat. Mech. Anal.}, 58 (1975)  pp. 207-218.


		\bibitem{DMR}  F.  Da  Lio, L. Martinazzi and T. Rivi\`ere, Blow-up Analysis of a nonlocal Liouville-type equation. {\em Analysis \& PDE}  8 (2015) no. 7, 1757-1805.
		
		\bibitem{dyg} E. N. Dancer, Y.  Du and Z. Guo, Finite Morse index solutions of an elliptic equation with supercritical exponent. {\em  J. Differential Equations}  250 (2011) 3281-3310.

		
		\bibitem{df} E. N. Dancer and A. Farina, On the classification of solutions of $-\Delta u=e^u$ on $R^n$: stability outside a compact set and applications. {\em Proc. Am. Math. Soc.} 137 (4) (2009) 1333-1338.
		
		
		\bibitem{ddw} J. D\'avila, L. Dupaigne and J. Wei, On the fractional Lane-Emden equation. {\em Trans. Amer. Math. Soc.} 369 (2017), no. 9, 6087-6104.
		
		\bibitem{ddww}  J. D\'avila, L. Dupaigne, K. L. Wang and J. Wei, A monotonicity formula and a Liouville-type theorem for a fourth order supercritical problem. {\em Adv. Math.} 258 (2014), 240-285.
		
		\bibitem{dgw} Y. Du,  Z. Guo and K.  Wang, Monotonicity formula and $\epsilon$-regularity of stable solutions to supercritical problems and applications to finite Morse index solutions. {\em Calc. Var. Partial Differential Equations}  50  (2014),  no. 3-4, 615?38.
		
		
		\bibitem{dggw}  L. Dupaigne, M. Ghergu, O. Goubet and G. Warnault, The Gel'fand problem for the biharmonic operator. {\em Arch. Ration. Mech. Anal.} 208 (2013), no. 3, 725-752.
	
	    \bibitem{Fabes} E. B. Fabes, C. E. Kenig and R. P. Serapioni, The local regualarity of solutions of degenerate elliptic equations. {\em Comm. Partial Differential Equations} 7 (1982), no. 1, 77-116.
		
		\bibitem{fa} M. M. Fall, Semilinear elliptic equations for the fractional Laplacian with Hardy potential. {\em Nonlinear Analysis}  193 (2020), 111311.
		
		\bibitem{f} A. Farina, On the classification of solutions of the Lane-Eˆmden equation on unbounded domains of $R^n$. {\em J. Math. Pures Appl.} 87 (2007) 537-561.
		
		\bibitem{f2} A. Farina, Stable solutions of $-\Delta u=e^u$ on $R^n$. {\em C. R. Math. Acad. Sci. Pari} 345, 2 (2007), 63-66.
		
		\bibitem{fs} M. Fazly and H. Shahgholian,  Monotonicity formulas for coupled elliptic gradient systems with applications.  {\em Advances in Nonlinear Analysis} 9 (2020) pp. 479-495.

		
		\bibitem{fw} M. Fazly and  J. Wei, On finite Morse index solutions of higher order fractional Lane-Emden equations. {\em Amer. J. Math.} 139, 2 (2017), 433-460.
		
		\bibitem{fw1} M. Fazly, and J.  Wei, On stable solutions of the fractional H\'{e}non-Lane-Emden equation. {\em Commun. Contemp. Math.} 18 (2016), no. 5, 1650005, 24 pp.
		
		\bibitem{fwy} M. Fazly, J.  Wei and W. Yang, Classification of finite Morse index solutions of higher-order Gelfand-Liouville equation. Preprint, ArXiv:2006.06089.
		
		\bibitem{fy} M. Fazly and W. Yang, On stable and finite Morse index solutions of the fractional Toda system. Preprint, ArXiv:2007.00069.
		
	\bibitem{gs} 	B. Gidas and J. Spruck, Global and local behavior of positive solutions of nonlinear elliptic equations. {\em Comm. Pure Appl. Math.} 34 (4) (1981) 525-598.
		
        \bibitem{HKP} P. Hajlasz and P. Koskola, Sobolev Met Poincare. {\em Mem. Amer. Math. Soc.} 145 (2000), no. 688, 106 pp.		
		
		\bibitem{h} I. W. Herbst, Spectral theory of the operator $(p^2+m^2)^{1/2}-Ze^2/r$. {\em Comm. Math. Phys.} 53(3) (1977) 285-294.
		
		\bibitem{hx} X. Huang, Stable weak solutions of weighted nonlinear elliptic equations. {\em Commun. Pure Appl. Anal.} 13 (2014), no. 1, 293-305.
		
		\bibitem{ha} A. Hyder, Structure of conformal metrics on $R^n$ with constant $Q$-curvature. {\em Differential Integral Equations} 32 (2019), no. 7-8, 423-454.
		
    	\bibitem{hy} A. Hyder and W. Yang, Classification of stable solutions to a non-local Gelfand-Liouville equation. {\em To appear in IMRN}.  	
		
		\bibitem{jl} D. D. Joseph and T. S. Lundgren, {Quasilinear Dirichlet problems driven by positive sources}.  {\em Arch. Rational Mech. Anal.}  49 (1972/73), 241-269.
		
				
		
\bibitem{PRP} C. Perez and E. Rela, Degenerate Poincare-Sobolev inequalities. {\em Trans. Amer. Math. Soc.} 372 (2019), 6087-6133.
		
\bibitem{ps} Q. H. Phan and Ph. Souplet, Liouville-type theorems and bounds of solutions of Hardy-H\'{e}non equations.  {\em J. Differential Equation} 252 (2012), 2544-2562.

\bibitem{r1} X. Ros-Oton and J. Serra, The extremal solution for the fractional Laplacian. {\em Calc. Var. Partial Differential Equations} 50 (2014), no. 3-4, 723-750.
		
		
\bibitem{wy}  C. Wang and  D. Ye, Some Liouville theorems for H\'{e}non type elliptic equations. {\em J. Funct. Anal.} 262 (2012), no. 4, 1705-1727.
		
		
		\bibitem{w0} K. Wang, Partial regularity of stable solutions to the supercritical equations and its applications. {\em Nonlinear Anal.} 75 (2012), no. 13, 5328-5260.
		
		\bibitem{w1} K. Wang, Partial regularity of stable solutions to the Emden equation. {\em Calc. Var. Partial Differential Equations} 44 (2012), no. 3-4, 601-610.
		
		
		\bibitem{w3} K. Wang, Stable and finite Morse index solutions of Toda system. {\em J. Differential Equations} 268 (2019), no. 1, 60-79.
		
		\bibitem{y} D. Yafaev, Sharp constants in the Hardy-Rellich inequalities. {\em J. Funct. Anal.} 168.1 (1999), 121-144.
	\end{thebibliography}
\end{document}